\documentclass{amsart}
\usepackage{amsmath}
\usepackage{amssymb}
\usepackage{graphicx}
\usepackage{tikz}
\usepackage{color}
\usepackage{enumerate}
\usepackage{hyperref}
\newtheorem{theorem}{Theorem}[section]

\newtheorem{conjecture}{Conjecture}
\newtheorem{remark}{Remark}
\newtheorem{question}{Question}
\newtheorem{lemma}[theorem]{Lemma}
\newtheorem{corollary}[theorem]{Corollary}
\newtheorem{proposition}[theorem]{Proposition}
\theoremstyle{definition}

\begin{document}

	\title[On Locally Generalized radical linear groups over rings ]{On Locally Generalized radical linear groups over rings} 
	
	
	\maketitle
	\begin{center}{LE VAN CHUA}\end{center}
	\begin{center}	
		\tiny{(1) An Giang University, Vietnam\\(2) Vietnam National University, Ho Chi Minh City, Vietnam\\
			e-mail: lvchua.tag@moet.edu.vn\\
			ORCID: 0000-0002-2443-4300}
	\end{center}
	
	\begin{center} {BUI XUAN HAI\footnote{Corresponding author}}
	\end{center}
	\begin{center}	
		\tiny{(1) Faculty of Mathematics and Computer Science, University of Science, Ho Chi Minh City, Vietnam\\(2) Vietnam National University, Ho Chi Minh City, Vietnam\\
			e-mail:  bxhai@hcmus.edu.vn\\
			ORCID: 0000-0002-4208-7883} 
	\end{center} 
		
\begin{abstract} In this paper, we investigate the structure of locally generalized radical linear groups over certain non-commutative rings and algebras. The base rings and algebras we consider include division rings, left artinian rings, locally finite algebras over fields, and PI-algebras over fields. Among various results, we get the positive answers to the General Burnside Problem and to Baer's Conjecture for some particular cases of linear groups. 
\end{abstract}
	
\tableofcontents
	
\section{Introduction and preliminaries}\label{Intro}

	The structure of locally generalized radical skew linear groups was investigated in our recent paper \cite{Pa_ChuaHai_2026}. Our aim in this paper is to study the structure of locally generalized radical linear groups over certain non-commutative rings and algebras including division rings. As it was remarked in  \cite{Pa_ChuaHai_2026}, the class of generalized radical groups is very large as it includes the class of nearly radical groups, the class of radical groups,
	the class of hyperabelian groups, the class of locally nilpotent groups, the class of solvable
	groups, the class of locally finite groups, and the class of solvable-by-finite groups. For definitions of these classes of groups and some basic properties of generalized radical groups we refer to \cite{Pa_ChuaHai_2026} and the references therein. 
	
Let $R$ be an associative ring with identity $1\ne 0$. The set of all invertible elements in $R$ constitutes the group under multiplication. 
We denote this group by $R^\times$, and call it the \textit{unit group} of $R$. If $R=D$ is a division ring, then by tradition, we use the symbol $D^*$ to indicate the unit group of $D$, and we often say about it as the \textit{multiplicative group} of $D$. If $n$ is a positive integer, then the unit group of the matrix ring $\mathrm{M}_n(R)$, denoted by $\mathrm{GL}_n(R)$, is called the \textit{general linear group} of degree $n$ over $R$. Every subgroup of $\mathrm{GL}_n(R)$ is a \textit{linear group} of degree $n$ over $R$. If $R=D$ is a division ring, then we say about $\mathrm{GL}_n(D)$ and its subgroups as the \textit{general skew linear group} and \textit{skew linear groups} of degree $n$ over $D$ respectively. 
	
This paper can be considered a continuation of our recent research in \cite{Pa_ChuaHai_2026}. Therefore, throughout it, most of the concepts and symbols we use are consistent with those we used in \cite{Pa_ChuaHai_2026}. Consequently, we refer to \cite{Pa_ChuaHai_2026} for the definitions and basic properties of the topics we study in the current paper, which is organized as follows. Apart from Section~ \ref{Intro}, where some notions and terminologies are recalled, and the main subjects are introduced, the paper contains also seven other sections. 
	
In Section \ref{generalized}, we continue our study of the structure of generalized radical skew linear groups that we began in \cite{Pa_ChuaHai_2026}. Let $D$ be a division ring, $n$ a positive integer, and $G$ an almost subnormal subgroup of the general skew linear group $\mathrm{GL}_n(D)$. In our recent work \cite{Pa_ChuaHai_2026}, it was conjectured \cite[Conjecture 1]{Pa_ChuaHai_2026} that if $G$ is locally generalized radical, then $G$ is central. The positive answer for this conjecture was obtained for $n\ge 2$ provided that $D$ is not a locally finite field (see \cite[Theorem 3.3]{Pa_ChuaHai_2026}) while in case $n=1$, the conjecture was solved affirmatively only for weakly locally finite division rings (see \cite[Theorem 4.11]{Pa_ChuaHai_2026}). Here, we shall give the positive answer to \cite[Conjecture 1]{Pa_ChuaHai_2026} in case $n=1$ for an arbitrary division ring $D$ (see Theorem~\ref{th:2.5}), and so we successfully complete the study of this conjecture. 
	
Section \ref{S3} is devoted to the study of  periodic radicals in skew linear groups and some problems related with the following conjecture posed by Herstein in 1978.
	\begin{conjecture}\label{conj:1}{\rm \cite[Conjecture 3]{Pa_Herstein_1978}}
		Let $D$ be a division ring with center $F$, and let $G$ be a subnormal subgroup of $D^*$. If $G$ is radical over $F$, then $G$ is central.
	\end{conjecture}
    Recall that an element $x\in D$ is \textit{radical} over $F$ if there exists some positive integer $m(x)$ depending on $x$ such that $x^{m(x)}\in F$. A non-empty subset $S$ of $D$ is \textit{radical} over $F$ if every its element is radical over $F$. Note that Conjecture \ref{conj:1} remains unsolved in general although it is true in several particular cases. Note also, Herstein himself \cite[Theorem~8]{Pa_Herstein_1978} proved that if $G$ is a periodic subnormal subgroup of the multiplicative group $D^*$ then $G$ is central. In Theorem \ref{th:3.2}, we show that the same result as in \cite[Theorem~8]{Pa_Herstein_1978} holds if $G$ is a periodic almost subnormal (instead of subnormal) subgroup of the group $D^*$. Further, we use this theorem to study the periodic radicals of almost subnormal subgroups in multiplicative groups of division rings in order to  describe the structure of (locally generalized radical)-by-periodic almost subnormal subgroups in division rings (see Theorem \ref{th:3.13}). Also, using the result obtained in Theorem \ref{th:2.5} of the previous section on the structure of generalized radical almost subnormal subgroups in division rings, we can prove that every periodic-by-(locally generalized radical) almost subnormal subgroup of the multiplicative group of a division ring must be central (see Theorem \ref{th:3.14}). This result strongly extends Herstein's result in \cite[Theorem 8]{Pa_Herstein_1978} we have mentioned above. 
    
    In Section \ref{S4}, we continue to study the structure of periodic skew linear groups we have began in Section \ref{S3}. The difference is that we consider ascendant subgroups instead of almost subnormal subgroups in general skew linear groups. Namely, we pose the following conjecture, which is more general than Conjecture \ref{conj:1} to study.
    \begin{conjecture}\label{conj:2}
    	Let $D$ be a division ring with center $F$, and let $G$ be an ascendant subgroup of $D^*$. If $G$ is radical over $F$, then $G$ is central.
    \end{conjecture}
	
The first main result we get is Theorem \ref{th:4.1}, which says that if $G$ is a periodic ascendant subgroup of $D^*$ then $G$ is central, giving the affirmative answer to \cite[Question 1]{Pa_ChuaHai_2026}. Further, using this theorem we can prove that every periodic-by-(generalized radical) ascendant subgroup of $D^*$ must be central (see Theorem~\ref{th:4.7}). Further, we consider periodic ascendant subgroups of the general linear group $\mathrm{GL}_n(D)$ of degree $n\ge 1$. We prove that every periodic locally generalized radical ascendant subgroup of the group $\mathrm{GL}_n(D)$ is central provided that $D$ is not a locally finite field in case $n\ge 2$ (see Theorem \ref{th:4.9}). Also, we get the positive answer for Conjecture~\ref{conj:2} in case $D$ is a weakly locally finite division ring (see Theorem~\ref{th:4.4}) or the center $F$ of $D$ is a locally finite field (see Theorem~ \ref{th:4.5}). Besides the new obtained results above, there is an useful  result which gives some information to verify Conjecture~\ref{conj:2}. Namely, we prove  that if an ascendant subgroup $G$ of $D^*$ is radical over $F$, then every periodic element of $G$ must be contained in $F$ (see Theorem~\ref{th:4.3}). This theorem strongly extends a result of Herstein in~\cite[Theorem~ 9]{Pa_Herstein_1978}.

We devote three next sections to study the structure of locally generalized radical linear groups over certain non-commutative rings and algebras such as left artinian rings, locally finite algebras, and PI-algebras. In particular, we investigate the General Burnside Problem for these groups.  Recall that in 1902, William Burnside \cite{Pa_Burn_1902} posed the following question: \begin{center}
	\textit{Is an arbitrary finitely generated periodic group finite?}
\end{center} 
Later, this question was termed the \textit{General Burnside Problem} (GBP for short).	Clearly, this is equivalent to ask whether every periodic group is locally finite.  It is well-known that the negative answer to GBP was given in 1964 by E. S. Golod and I. R. Shafarevich  who constructed  \cite{Pa_Golod_1964, Pa_Golod-Shaf_1964} a family of infinite finitely generated $p$-groups for an arbitrary prime number $p$. Although in general, the problem was solved negatively, there exist various affirmative answers to it in particular cases. For instance, recall that the GBP for linear groups over fields was studied by W. Burnside and I. Schur. From Burnside's First Theorem (see e.g. \cite[(9.4)]{Bo_Lam_2001}), we see that a periodic linear group of finite exponent over a field of characteristic zero must be finite. Later, Schur \cite{Pa_Schur_1911} proved  that every finitely generated periodic linear group over a field is finite (also see~\cite[(9.9)]{Bo_Lam_2001}), answering positively to the GBP for linear groups over fields. Also, C. Procesi~\cite{Pa_Procesi_1966} and A. Tokarenko~\cite{Pa_Tokarenko_1968} obtained an affirmative answer to GBP for linear groups over $\mathrm{PI}$-algebras.

Let $R$ be a ring, $n$ a positive integer, and $\mathrm{GL}_n(R)$ the general linear group over $R$. In Section \ref{S5}, we study the structure of locally generalized radical subgroups of $\mathrm{GL}_n(R)$. The main result we get is Theorem \ref{th:5.7}, which gives in particular the affirmative answer to GBP for almost subnormal subgroups of the general linear group $\mathrm{GL}_n(R)$ over a left artinian ring $R.$

In Section \ref{S6}, we study linear groups over locally finite algebras. Namely, let $A$ be an $F$-algebra, where $F$ is a field. Recall that $A$ is \textit{locally finite} if every finite subset of elements of $A$ generates a finite-dimensional subalgebra. Considering subgroups of the general linear group $\mathrm{GL}_n(A)$, we show that every locally generalized radical subgroup of $\mathrm{GL}_n(A)$ over a locally finite algebra $A$ is locally (solvable-by-finite) (see Theorem~\ref{th:6.8}). As an application, we can prove that the GBP is solved positively for linear groups over locally finite algebras (see Theorem~\ref{th:6.11}). As a consequence, the GBP is solved positively for linear groups over weakly locally finite division rings (see Corollary~\ref{cor:6.12}).

In Section \ref{S7}, we study the structure of locally generalized radical linear groups over $\mathrm{PI}$-algebras. It is shown that if $G$ is a locally generalized radical linear group over a $\mathrm{PI}$-algebra $A$, then $G$ contains a locally nilpotent normal subgroup $N$ such that the quotient group $G/N$ is locally (solvable-by-finite). Moreover, if $G$ is finitely generated, then $G$ has a normal series: $1\trianglelefteq N\trianglelefteq H\trianglelefteq G$ such that $N$ is locally nilpotent, $H/N$ is solvable, and $G/H$ is finite (see Theorem~\ref{th:7.2}). Besides, we prove that if $A$ is an affine $\mathrm{PI}$-algebra over a field, then every locally generalized radical linear group over $A$ is locally (solvable-by-finite) (see Theorem~\ref{th:7.4}).

In the end, we devote the last section, Section \ref{S8} to study Baer's Conjecture.
Recall that a group $G$ is said to be \textit{noetherian} if it satisfies the maximal condition on subgroups, or, equivalently, if every its subgroup is finitely generated (see~\cite{Zassenhaus_1969}). It is easy to see that polycyclic-by-finite groups are noetherian. In~\cite{Pa_Baer_1956} Baer conjectured that every noetherian group is polycyclic-by-finite. It turned out that in general, this conjecture is not true. Indeed, according to \cite[Theorem 28.3]{Bo_Ol'shanskii_1991}, there is a simple torsion-free $2$-generated group $G$ in which every proper subgroup is infinite cyclic. Clearly, such a group $G$ is noetherian, and  being a simple group, $G$ is not polycyclic-by-finite. However, 
in \cite[Theorem 1]{Zassenhaus_1969}, Zassenhaus proved that every noetherian linear group over a field is polycyclic-by-finite. Inspired by this result, in this section we will look for the conditions under which Baer's Conjecture holds. Firstly, in Theorem \ref{th:8.3}, we show that every generalized radical noetherian group is polycyclic-by-finite. Further, we extend Zassenhaus' result in \cite[Theorem 1]{Zassenhaus_1969} by proving that every noetherian linear group over a locally finite algebra is polycyclic-by-finite (see Theorem~\ref{th:8.7}). Finally, we finish the section by giving the affirmative answer to Baer's Conjecture for linear groups over $\mathrm{PI}$-algebras (see Theorem~\ref{th:8.9}).

\section{Free subgroups of almost subnormal subgroups \\ of multiplicative groups of  division rings}\label{generalized}

Let $D$ be a division ring, $G$ a subgroup of the multiplicative group $D^*$. The main aim of this section is to prove that if $G$ is  locally generalized radical almost subnormal in $D^*$, then $G$ is contained in the center $F$ of $D$. Recall that in 1950, L. K. Hua~\cite{Pa_Hua_1950} proved that if the group $D^*$ is solvable, then $D$ is a field. 
Seven years later, W. R. Scott~\cite{Pa_Scott_1957}  extended Hua's result by proving that if $G$ is a solvable normal subgroup of $D^*$ then $G$ is central. Further, the result of Scott has been extended for $G$ is solvable subnormal in $D^*$  by Stuth \cite[Theorem 4]{Pa_Stuth_1964} (1964); for $G$ is locally solvable normal in $D^*$  by Zalesskii \cite[Theorem 2]{Pa_Zalesskii_1965} (1965);  for $G$ is locally solvable subnormal in $D^*$ by Danh and Khanh \cite[Theorem 1]{Pa_Danh-Khanh_2021} (2021); for $G$ is (locally solvable)-by-(locally finite) almost subnormal in $D^*$ by Chua, Nam, and Hai \cite[Theorem 2.6]{Pa_ChNaHa_24} (2024); and for $G$ is locally (solvable-by-finite) almost subnormal in $D^*$ by Hai and Chua \cite[Theorem 3.9]{Pa_HaiChua_2025} (2025). 
Finally, recently in \cite{Pa_ChuaHai_2026}, we have proved a result, which extends all results mentioned above. Namely, \cite[Theorem 3.5]{Pa_ChuaHai_2026} showed that every generalized radical almost subnormal subgroup of $D^*$ is central. It turns out that the result of \cite[Theorem 3.5]{Pa_ChuaHai_2026} still can be extended.
Namely, in this section, we shall prove that 
every locally generalized radical almost subnormal subgroup of $D^*$ is central (see Theorem \ref{th:2.5}). Note that Theorem \ref{th:2.5} gives the positive answer to \cite[Conjecture 1]{Pa_ChuaHai_2026} in the remaining case $n=1$, and so the study of this conjecture is now successfully complete.

To get Theorem \ref{th:2.5}, we shall prove a more general result on the existence of non-cyclic free subgroups in  locally 
generalized radical almost subnormal subgroup of $D^*$.
Observe that the question of whether the multiplicative group of any non-commutative division ring contains a non-cyclic free subgroup posed in 1977 by Lichtman \cite{Pa_Li_77} remains now still open in general. Let $D$ be a division ring, and $G$ a subgroup of the multiplicative group $D^*$. In \cite[Theorem 2]{Pa_Li_78}, Lichtman proved that if $G$ is normal in $D^*$  and $G$ contains a non-abelian nilpotent-by-finite subgroup, then $G$ contains a non-cyclic free subgroup. The same result was obtained by Danh and Deo \cite[Theorem~1.2]{Pa_Danh-Deo_2023} if $G$ is subnormal in $D^*$ and $G$ contains a non-abelian solvable subgroup; by Ngoc, Bien, and Hai \cite[Theorem 4.7]{Pa_nbh_17} if $G$ is an almost subnormal subgroup {\color{blue}containing non-abelian nilpotent-by-finite subgroup} in $D^*$ provided that $D$ is algebraic over its center $F$; by Hai and Chua \cite[Theorem 3.5]{Pa_HaiChua_2025} if $G$ is almost subnormal in $D^*$ and $G$ contains a non-abelian locally (solvable-by-finite) subgroup. Here, in Theorem~ \ref{th:2.3}, we prove a very general result, which extends all results mentioned above. Namely, we prove that if $G$ contains a non-abelian locally generalized radical subgroup, then $G$ contains a non-cyclic free subgroup. Clearly, Theorem~\ref{th:2.5} is a direct consequence of Theorem~ \ref{th:2.3}. 

\begin{lemma}\label{lem:2.1}
Let $G$ be a group. Then, the following statements hold:
\begin{itemize}
\item [(1)] If $G$ is abelian-by-(locally finite), then $G$ is locally (abelian-by-finite).
\item [(2)] If $G$ is abelian-by-(locally nilpotent), then $G$ is locally solvable.
\end{itemize}
\end{lemma}
\begin{proof}
(1) Suppose that $G$ is abelian-by-(locally finite). Then, there exists an abelian normal subgroup $A$ of $G$ such that $G/A$ is locally finite. Let $N$ be a finitely generated subgroup of $G$. Then, $NA/A$ is a finitely generated subgroup of $G/A$, and so it is finite. The isomorphism $$NA/A\cong N/N\cap A$$ shows that $N$ is abelian-by-finite. Hence, $G$ is locally (abelian-by-finite).
		
(2) Suppose that $G$ is abelian-by-(locally nilpotent). Then, there exists an abelian normal subgroup $A$ of $G$ such that $G/A$ is locally nilpotent. Let $N$ be a finitely generated subgroup of $G$. Then, $NA/A$ is a finitely generated subgroup of $G/A$, and so it is nilpotent. From the isomorphism $$NA/A\cong N/N\cap A,$$ it follows that $N$ is solvable. Hence, we conclude that $G$ is locally solvable.
\end{proof}
\begin{lemma}\label{lem:2.2} If $G$ is a non-abelian generalized radical group, then $G$ contains a non-abelian locally (solvable-by-finite) subgroup.
\end{lemma}
\begin{proof} Let 
	$$1=G_0\le G_1\le\cdots G_\alpha\le G_{\alpha+1}\le\cdots G_\gamma=G$$
be an ascending normal series in $G$ of length $\gamma$  whose factors are locally finite or locally nilpotent. We shall prove the lemma by transfinite induction on $\gamma.$ If $\gamma=1,$ then $G$ is either locally finite or locally nilpotent. We see that in both cases, $G$ is locally (solvable-by-finite), and so the result follows. Now, let $\gamma >1$. Observe that $G_\alpha$ is generalized radical for all $\alpha$. Suppose by induction that the conclusion of the lemma holds for all $\alpha<\gamma$ such that $G_\alpha$ is non-abelian. We shall prove that the conclusion also holds for $G$. This is clear if there exists an $\alpha<\gamma$ such that $G_\alpha$ is non-abelian. Now, suppose that $G_\alpha$ is abelian for all $\alpha<\gamma$.
If $\gamma$ is a limit ordinal, then $G=\bigcup_{\alpha<\gamma}G_\alpha$. Then, for any two elements $x,y\in G$, there exist some $\alpha<\gamma$ such that $x,y\in G_\alpha$, and so $xy=yx$. Thus, $G$ is abelian, a contradiction. Hence, $\gamma$ is not a limit ordinal, and so the ordinal $\gamma-1$ exists. Then, $G_{\gamma-1}$ is an abelian normal subgroup of $G$, and $G/G_{\gamma-1}$ is either locally finite or locally nilpotent. In view of Lemma \ref{lem:2.1},  $G$ is either locally (abelian-by-finite) or locally solvable. Clearly, in both cases, the group $G$ satisfies the requirement of the lemma. 
Hence, the proof of the lemma is now complete.
\end{proof}

We are now in the position to prove the following theorem which can be considered as the main result of this section.
\begin{theorem}\label{th:2.3} Let $D$ be a division ring, and $G$ an almost subnormal subgroup of the multiplicative group $D^*$. If $G$ contains a non-abelian locally generalized radical subgroup, then $G$ contains a non-cyclic free subgroup.
\end{theorem}
\begin{proof} Suppose that $G$ is an almost subnormal subgroup of $D^*$ containing a non-abelian locally generalized radical subgroup, say $H$. Take non-commuting elements $a,b\in H$, and set $N=\langle a, b\rangle$, the subgroup of $H$ generated by $a$ and $b$. Then, $N$ is a non-abelian generalized radical group, and by Lemma~\ref{lem:2.2}, $N$ contains a non-abelian locally (solvable-by-finite) subgroup. According to~\cite[Theorem 3.5]{Pa_HaiChua_2025}, $N$ contains a non-cyclic free subgroup, and so does $G$. 
\end{proof}
Before proving the next theorem, it is useful to make a simple observation as follows. Let $G$ be a subgroup of the multiplicative group $D^*$ of a division ring $D$. If $G$ is a non-abelian, then it is obvious that $G$ is non-central. However, non-central abelian subgroups of $D^*$ may exist. For instance, if $a\in D\backslash F$ then the cyclic subgroup $\langle a\rangle$ generated by the element $a$ is non-central, where $F$ is the center of $D$. Now, suppose that $G$ is non-central. We can show that if $G$ is almost subnormal in $D^*$, then $G$ is non-abelian. Indeed, by \cite[Proposition 2.2 ]{Pa_Khanh-Hai_2022}, $G$ contains a non-central subnormal subgroup, say $H$ of $D^*$. If $G$ is abelian, then so is $H$. In view of \cite[Theorem 4]{Pa_Stuth_1964}, $H$ is central, a contradiction. For the convenience of use, let us state this fact as a lemma as follows.
\begin{lemma}\label{lemma 2.4}
Let $D$ be a division ring, and $G$ an almost subnormal subgroup of the multiplicative group $D^*$. Then, $G$ is non-central if and only if $G$ is non-abelian.
\end{lemma}
Due to Theorem \ref{th:2.3} and Lemma \ref{lemma 2.4}, we can get the following result, which completes the proof of \cite[Conjecture 1]{Pa_ChuaHai_2026} in the remaining case $n=1$.
\begin{theorem}\label{th:2.5} Let $D$ be a division ring, and $G$ an almost subnormal subgroup of the multiplicative group $D^*.$ If $G$ is locally generalized radical, then $G$ is central.
\end{theorem}
\begin{proof} Suppose that $G$ is a locally generalized radical almost subnormal subgroup of the multiplicative group $D^*.$ We have to show that $G$ is central. Suppose by contrary that $G$ is non-central. Then, by Lemma \ref{lemma 2.4}, $G$ is non-abelian. Take some noncommuting elements $a, b\in G$, and let $H=\langle a, b\rangle$ be the subgroup of $G$ generated by $a$ and $b$. From the hypothesis, it follows that $H$ is a non-abelian generalized radical group, and so by Theorem \ref{th:2.3}, $G$ contains a non-cyclic free subgroup. But, this contradicts to~\cite[Theorem 2.4]{Pa_ChuaHai_2026}. Hence, $G$ is central as required to prove. 
\end{proof}
\begin{corollary}\label{cor:2.6}	
Let $D$ be a non-commutative division ring with the center $F$, and $G$ a subgroup of the derived group $D'$ of the multiplicative group $D^*$. If $G$ is a locally generalized radical almost subnormal subgroup of $D'$, then $G$ is contained in the center $Z(D')$. Moreover,  the following assertions hold:
   \begin{itemize}
	  \item [(1)] The central factor group $D^*/F^*$ contains no non-trivial locally generalized radical almost subnormal subgroups.
	  \item [(2)] The central factor group $D'/Z(D')$ contains no non-trivial locally generalized radical almost subnormal subgroups.
	  \item [(3)] The derived subgroup $(D^*/F^*)'$ contains no non-trivial locally generalized radical almost subnormal subgroups.
    \end{itemize}
\end{corollary}
\begin{proof} Assume that $G$ is a locally generalized radical almost subnormal subgroup of $D'$. Then, $G$ is also a locally generalized radical almost subnormal subgroup of the multiplicative group $D^*$. Hence, by Theorem \ref{th:2.5}, $G\le F^*,$ and so by~\cite[Lemma~ 3.1]{Pa_ChuaHai_2026}, $G\leq D'\cap F^*=Z(D')$. Suppose that $D^*/F^*$ contains a non-trivial locally generalized radical almost subnormal subgroup, say $L$. Denote by $H$ the preimage of $L$ via the natural epimorphism $D^*\longrightarrow D^*/F^*$.  Then, $H$ is non-central almost subnormal in $D^*.$ By~\cite[Lemma 2.1]{Pa_ChuaHai_2026}, $H$ is locally generalized radical. Since this fact contradicts to Theorem~\ref{th:2.5}, the conclusion of (1) follows. The conclusion of (2) follows by the similar argument as in the proof of (1). The final assertion follows from the fact that the derived subgroup $(D^*/F^*)'$ is isomorphic to the group $D'/Z(D')$ (see~\cite[Lemma 3.1]{Pa_ChuaHai_2026}) and the conclusion of (2). 
\end{proof}

\section{Periodic skew linear groups}\label{S3}
Let $G$ be any group. The set of all periodic elements in $G$ denoted by $T(G)$ may not be a subgroup of $G$. However, periodic normal subgroups of $G$ always exist, and a unique maximal periodic normal subgroup of $G$, denoted by $Tor(G)$, is called the \textit{periodic radical} of $G$. It may happen that $Tor(G)=T(G)$, and in this case, we say that $\mathit{Tor}(G)$ is the \textit{torsion subgroup} of $G.$ For instance, if $G$ is a locally nilpotent group, then $\mathit{Tor}(G)=T(G)$ and the quotient group $G/\mathit{Tor}(G)$ is torsion-free (see, e.g. ~\cite[Proposition 1.2.11]{Bo_Dixon_2017}). In this section, we investigate the existence of torsion subgroups in skew linear groups and related problems. Let $D$ be a division ring, $n$ a positive integer, and $G$ an almost subnormal subgroup of the general linear group $\mathrm{GL}_n(D)$ of degree $n$ over $D$. We seek the necessary and sufficient conditions for the existence of the torsion subgroup $Tor(G)$ of the group $G$ (see Proposition \ref{pro:3.5}). In order to get these conditions, we have to use some previous result on periodic skew linear groups, especially of degree $1$.  Historically, in 1978, Herstein \cite[Theorem~8]{Pa_Herstein_1978} proved that if $G$ is a periodic subnormal subgroup of the multiplicative group $D^*$ then $G$ must be contained in the center $F$ of $D$. Here, basing on our recent result in \cite[Theorem 2.1]{Pa_ChNaHa_24}, it is easy to get the same result as in \cite[Theorem~8]{Pa_Herstein_1978} but for the case when $G$ is assumed to be an almost subnormal subgroup of the group $D^*$ (see Theorem \ref{th:3.2}). This theorem turns out useful in the study of periodic radicals of almost subnormal subgroups in multiplicative groups of division rings (see for example, Theorem \ref{th:3.12}). Further, we use the result obtained previously in 
Theorem~\ref{th:2.5} to describe the structure of 
periodic-by-(locally generalized radical) group of the multiplicative group of a division ring (see Theorem \ref{th:3.14}).
\begin{lemma}\label{lem:3.1} Let $D$ be a division ring that is not a locally finite field, and $n$ a positive integer. If $n\geq 2$, then the following assertions hold:
	\begin{enumerate}
		\rm\item\textit{The groups $\mathrm{SL}_n(D)$ is not periodic, and so is $\mathrm{GL}_n(D)$.}
		\item\textit{The group $\mathrm{PSL}_n(D)$ is not periodic.}
		\item\textit{The commutator subgroup $\mathrm{PGL}_n'(D)$ is not periodic, and so is $\mathrm{PGL}_n(D)$.}
	\end{enumerate}	
\end{lemma}
\begin{proof} 
	(1) By \cite[4.5.1]{Pa_ShWe_86}, the group  $\mathrm{SL}_n(D)$ contains a non-cyclic free subgroup, so it cannot be periodic. As a consequence, the group $\mathrm{GL}_n(D)$ is not periodic, too.
	
	(2) Assume that the group $\mathrm{PSL}_n(D)$ is periodic. Let $N$ be a  non-cyclic free subgroup of $\mathrm{SL}_n(D)$. Then, $N\cap Z(\mathrm{SL}_n(D))\triangleleft N$, and  $$N/(N\cap Z(\mathrm{SL}_n(D)))\cong NZ(\mathrm{SL}_n(D))/Z(\mathrm{SL}_n(D))$$
	is a periodic group. Since $N$ is a non-cyclic free group, the center $Z(N)=1,$ and it follows that $N\cap Z(\mathrm{SL}_n(D))=1.$ This implies that $N$ is periodic, a contradiction. 
	
	(3) Because $\mathrm{PGL}_n'(D)\cong \mathrm{PSL}_n(D),$ the statement is obvious in view of (2).
\end{proof}
\begin{remark} If $D$ is a locally finite field and $n$ a positive integer, then all the groups $$\mathrm{SL}_n(D),\ \mathrm{GL}_n(D),\ \mathrm{PSL}_n(D),\ \mathrm{PGL}_n(D),\ \mathrm{PGL}_n'(D)$$ are locally finite, and so they are periodic.
\end{remark}
\begin{theorem}\label{th:3.2} Let $D$ be a division ring, and $n$ a positive integer. Suppose that $G$ is an almost subnormal subgroup of the  general skew linear group $\mathrm{GL}_n(D)$. If $G$ is periodic, then $G$ is central provided that $D$ is not a locally finite field in case $n\ge 2$.
\end{theorem}
\begin{proof} Deny the theorem, suppose that $G$ is non-central. 
	
	\textit{Case 1. $n=1$:}
	
	Since $G$ is non-central, by \cite[Theorem 2.1]{Pa_ChNaHa_24}, $G$ contains a non-central subnormal subgroup $N$ of $D^*$. But this contradicts to a result of Herstein in~\cite[Theorem 8]{Pa_Herstein_1978}. 
	
	\textit{Case 2. $n\ge 2$:}
	
	In this case, $D$ is a not locally finite field. Since $G$ is non-central, in view of \cite[Theorem~3.3]{Pa_nbh_17}, $G$ contains  $\mathrm{SL}_n(D)$, and so $\mathrm{SL}_n(D)$ is a periodic group. This contradicts to the conclusion of Lemma~\ref{lem:3.1}.
\end{proof}
\begin{corollary}\label{cor:3.3} Let $D$ be a division ring, and $n$ a positive integer. If $H$ is a periodic almost subnormal subgroup of $\mathrm{SL}_n(D)$, then $H$ is contained in the center $Z(\mathrm{SL}_n(D))$ provided that $D$ is not a locally finite field in case $n\ge 2$. 
\end{corollary}
\begin{proof} Assume that $H$ is a periodic almost subnormal subgroup of the skew special linear group $\mathrm{SL}_n(D).$ Then, $H$ is a periodic almost subnormal subgroup of the skew general linear group $\mathrm{GL}_n(D).$ According to Theorem \ref{th:3.2}, $H$ is contained in the center $Z(\mathrm{GL}_n(D)).$ It is well-known that $$Z(\mathrm{SL}_n(D))=\mathrm{SL}_n(D)\cap Z(\mathrm{GL}_n(D)).$$ From the facts, it follows that $H$ is contained in the center $Z(\mathrm{SL}_n(D)).$
\end{proof}
\begin{theorem}\label{th:3.4} Let $D$ be a division ring, and $n$ a positive integer. Assume that one of the following two conditions holds:
	\begin{itemize}
		\item [(1)] $n=1$ and $D$ is weakly locally finite.
		\item [(2)] $n\geq 2$ and $D$ is not a locally finite field.
	\end{itemize}
	Then, the projective general skew linear group $\mathrm{PGL}_n(D)$ contains no non-trivial periodic almost subnormal subgroups. Consequently, the projective special skew  linear group $\mathrm{PSL}_n(D)$ contains no non-trivial periodic almost subnormal subgroups.
\end{theorem}
\begin{proof} Let us consider the natural homomorphism
	$$\varphi: \mathrm{GL}_n(D)\longrightarrow \mathrm{PGL}_n(D)=\mathrm{GL}_n(D)/Z(\mathrm{GL}_n(D)).$$
	Assume that $\overline{G}$ is a non-trivial periodic almost subnormal subgroup of  $\mathrm{PGL}_n(D)$. Denote by $G$ the preimage of $\overline{G}$ via this homomorphism; that is, $G=\varphi^{-1}(\overline{G})$. Then, $G$ is a non-central  almost subnormal subgroup of $\mathrm{GL}_n(D)$. Clearly, $G$ properly contains $Z(\mathrm{GL}_n(D))$, and $\overline{G}=G/Z(\mathrm{GL}_n(D))$.
	
	(1) Assume that $n=1$, and $D$ is a weakly locally finite division ring. Then, $\mathrm{PGL}_1(D)=D^*/F^*$, where $F$ is the center of $D$. Since $G$ is an almost subnormal subgroup of $D^*$, $Z(G)$ is an abelian almost subnormal subgroup of $D^*$. We claim that $Z(G)$ is contained in $F$. Indeed, if otherwise then by \cite[Theorem 2.1]{Pa_ChNaHa_24}, $Z(G)$ contains a non-central abelian subnormal subgroup, say $H$ of $D^*$. But, by Stuth's result \cite[Theorem 1]{Pa_Stuth_1964}, $H$ is central, a contradiction. Hence,  $Z(G)$ is contained in $F$,  which implies that $Z(G)=G\cap F$. It follows that
	$$\overline{G}=GF^*/F^*\cong G/G\cap F^*=G/Z(G).$$
	Since $G$ is a non-central almost subnormal subgroup in a weakly locally finite division ring $D$, in view of \cite[Theorem 4.2]{Pa_nbh_17}, $G$ contains a non-cyclic free subgroup, say $N$, and so the center  $Z(N)$ of $N$ is trivial. Now, if $\overline{G}$ is periodic, then so is $G/Z(G)$. Clearly, $N\cap Z(G)\le Z(N)=1$, so we have
	$$N=N/(N\cap Z(G))\cong NZ(G)/Z(G)\le G/Z(G),$$ 
	which implies that $N$ is periodic, that is impossible because $N$ is a free group. 
	
	(2) Assume that $n\ge 2$, and $D$ is not a locally finite field. Since $G$ is non-central, in view of~\cite[Theorem~3.3]{Pa_nbh_17}, $G$ contains $\mathrm{SL}_n(D),$ and so $\overline{G}$ contains $$\mathrm{SL}_n(D)Z(\mathrm{GL}_n(D))/Z(\mathrm{GL}_n(D))=\mathrm{PGL}_n'(D).$$ By Lemma~\ref{lem:3.1}, the derived subgroup $\mathrm{PGL}_n'(D)$ is not periodic. Hence, the group $\overline{G}$ is not periodic, a contradiction. 
	
	Finally, the last conclusion follows immediately from Lemma \ref{lem:3.1} and the isomorphism $\mathrm{PSL}_n(D)\cong \mathrm{PGL}_n'(D)$.
\end{proof}
In the remaining part of this section, we study the structure of periodic radicals of skew linear groups. Let $D$ be a division ring, and $n$ a positive integer. C. P. Coelho and C. P. Millies~\cite[Proposition 3.1]{Pa_CoeMil_1998} showed that $T(D^*)$ forms a subgroup of $D^*$ if and only if $T(D^*)$ is central; that is, $T(D^*)$ is contained in the center of $D,$ and that $T(\mathrm{GL}_n(D))$ form a subgroup of $\mathrm{GL}_n(D)$ if and only if $D$ is a field, which is an algebraic extension of its prime subfield provided that $D$ has positive characteristic in case $n\geq 2$ (see~\cite[Proposition 3.3]{Pa_CoeMil_1998} and~\cite[Proposition 2.2]{Pa_Milies_1981}). Assume that $G$ is an almost subnormal subgroup of the  general skew linear group $\mathrm{GL}_n(D).$ Here, we give necessary and sufficient conditions such that $T(G)$ forms a subgroup of $G.$ The following result shows that $T(G)$ is a subgroup of $G$ if and only if $T(G)$ is central provided that $D$ is not a locally finite field in case $n\geq 2.$

\begin{proposition}\label{pro:3.5} Let $D$ be a division ring, and $n$ a positive integer. Assume that $G$ is an almost subnormal subgroup of the general skew linear group $\mathrm{GL}_n(D).$ Suppose further that $D$ is not a locally finite field in case $n\geq 2.$ Then,
	\begin{itemize}
		\item [(1)] $T(G)$ is a subgroup of $G$ iff $T(G)$ is central. In case this, $T(G)=\mathit{Tor}(G)$ and the quotient group $G/\mathit{Tor}(G)$ is torsion-free.
		\item [(2)] If $G\leq \mathrm{SL}_n(D),$ then $T(G)$ is a subgroup of $G$ iff $T(G)\subseteq Z(\mathrm{SL}_n(D)).$ In case this, $T(G)=\mathit{Tor}(G)$ and the quotient group $G/\mathit{Tor}(G)$ is torsion-free.
	\end{itemize}
\end{proposition}
\begin{proof} (1) If $T(G)$ is a subgroup of $G,$ then clearly $T(G)$ is normal in $G.$ Consequently, $T(G)$ is almost subnormal in $\mathrm{GL}_n(D)$. In view of Theorem~\ref{th:3.2}, $T(G)$ is central. The converse is obvious.
	
(2) Assume that $G$ is a subgroup of $\mathrm{SL}_n(D).$ If $T(G)$ is a subgroup of $G,$ then $T(G)\subseteq Z(\mathrm{GL}_n(D))$ by the conclusion of (1), and therefore $$T(G)\subseteq Z(\mathrm{GL}_n(D))\cap \mathrm{SL}_n(D)=Z(\mathrm{SL}_n(D)).$$ The converse is now obvious.
\end{proof}
\begin{proposition}\label{pro:3.6} Let $D$ be a division ring, $n$ a positive integer, and $U$  an almost subnormal subgroup of the projective general skew linear group $\mathrm{PGL}_n(D).$ Suppose that one of the following two conditions holds:
	\begin{itemize}
		\item [(1)] $n=1$ and $D$ is weakly locally finite.
		\item [(2)] $n\geq 2$ and $D$ is not a locally finite field.
	\end{itemize}
	Then, $T(U)$ is a subgroup of $U$ iff $T(U)=1$. In case this, $U$ is a torsion-free group.
\end{proposition}
\begin{proof} If $T(U)$ is a subgroup of $U,$ then $T(U)$ is periodic almost subnormal in $\mathrm{PGL}_n(D).$ Thus, according to Theorem~\ref{th:3.4}, $T(U)=1,$ and so $U$ is torsion-free.
\end{proof} 
\begin{proposition}\label{pro:3.7} Let $D$ be a division ring, $n$ a positive integer, and $V$ an almost subnormal subgroup of the projective special skew linear group $\mathrm{PSL}_n(D).$ Assume that one of the following two conditions holds:
	\begin{itemize}
		\item [(1)] $n=1$ and $D$ is weakly locally finite.
		\item [(2)] $n\geq 2$ and $D$ is not a locally finite field.
	\end{itemize}
	Then, $T(V)$ is a subgroup of $V$ iff $T(V)=1$. In case this, $V$ is a torsion-free group.
\end{proposition}
\begin{proof} Assume that $T(V)$ is a subgroup of $V.$ Then, $T(V)$ is periodic almost subnormal in $\mathrm{PSL}_n(D).$ By Theorem~\ref{th:3.4}, $T(V)=1,$ and so $V$ is a torsion-free group.
\end{proof} 
Let $K$ be a division subring of a division ring $D,$ and $S$ a non-empty subset of the multiplicative group $D^*.$ We say that $K$ is \textit{$S$-invariant}  if $xKx^{-1}\subseteq K$ for all $x\in S.$ The Cartan-Brauer-Hua Theorem states that if $K$ is  $D^*$-invariant then either $K\subseteq F$ or $K=D.$ Clearly, $K$ is $D^*$-invariant if and only if the group $K^*$ is normal in $D^*$. In \cite{Pa_HersteinScott_1963}, Herstein and Scott proved the same result for $K$ but under a weaker assumption that $K^*$ is subnormal in $D^*$. Later, Stuth~\cite[Theorem 1]{Pa_Stuth_1964} obtained very general result by proving that if $K$ is $G$-invariant for some non-central subnormal subgroup $G$ of $D^*$, then either $K\subseteq F$ or $K=D$. Here, we make a simple additional observation that Stuth's Theorem can be extended as follows for the case when $G$ is assumed to be almost subnormal instead of subnormal in $D^*$.

\begin{lemma}\label{lem:3.8} Let $D$ be a division ring with the center $F$, $K$ is a division subring of $D$, and $G$ a non-central almost subnormal subgroup of the multiplicative group $D^*.$ If  $K$ is $G$-invariant, then either $K\subseteq F$ or $K=D.$
\end{lemma}
\begin{proof} According to~\cite[Theorem 2.1]{Pa_ChNaHa_24}, $G$ contains a non-central subnormal subgroup $H$ of $D^*$. Since $K$ is $G$-invariant, it follows that $K$ is also $H$-invariant, and by \cite[Theorem 1]{Pa_Stuth_1964}, either $K\subseteq F$ or $K=D.$
\end{proof}
\begin{lemma}\label{lem:3.9} Let $D$ be a division ring with the center $F.$ Assume that $G$ is an almost subnormal subgroup of $D^*.$ Then,
	\begin{itemize}
		\item [(1)] If $G$ is solvable, then $G$ is central.
		\item [(2)] $Z(G)=G\cap F^*=G\cap C_{D^*}(G),$ and so $C_{D}(G)=F.$
	\end{itemize}	
\end{lemma}
\begin{proof} (1) Suppose that $G$ is a solvable almost subnormal subgroup of $D^*$. If $G$ is non-central, then by \cite[Theorem 2.1]{Pa_ChNaHa_24}, $G$ contains a non-central  subgroup, say $H$, which is subnormal in $D^*$. Since $H$ is also solvable, in view of \cite[Theorem 4]{Pa_Stuth_1964}, $H$ is central, and we are in a contradiction. Hence, $G$ must be central.
	
	(2) By conclusion of Part (1), $Z(G)\le F^*,$ and so $Z(G)\leq F^*\cap G\leq Z(G).$ We conclude that $Z(G)=G\cap F^*.$ Observe that $Z(G)=G\cap C_{D^*}(G)$ and $F^*\leq C_{D^*}(G).$ From the facts, it follows that $Z(G)=G\cap F^*=G\cap C_{D^*}(G),$ and so $C_{D}(G)=F.$
\end{proof}
\begin{lemma}\label{lem:3.10} Let $D$ be a division ring with center $F.$ Assume that $G$ a non-central almost subnormal subgroup of the multiplicative group $D^*$. Then, the following statements hold:
	\begin{enumerate}
		\rm\item\textit{If $a\in D\backslash F,$ there exists an element $g\in G$ such that $(a^G)^*=D.$} 
		\rm\item\textit{If $a\in D\backslash F,$ there exists an element $g\in G$ such that $[g,a]$ is not in $F.$}
	\end{enumerate}
\end{lemma}
\begin{proof} $(1)$ Observe that $(a^G)^*$ is invariant under $G$. By Lemma~\ref{lem:3.8}, $D=(a^G)^*.$
	
	$(2)$ Suppose by contrary that $[g,a]\in F$ for every $g\in G.$ One has $a^g=[g,a]a$, so $a^g\in C(a).$  Hence, $C(a)$ contains $a^G$ and so $C(a)$ contains $(a^G)^*.$ By the result of $(1)$, it follows that $C(a)=D,$ from which we have $a\in F,$ a contradiction. Hence, we conclude that $[g,a]$ is not in $F.$ This completes the proof of the lemma.
\end{proof}
\begin{lemma}\label{lem:3.11} Let $D$ be a division ring with center $F.$ Assume that $a\in D$ is a non-central periodic element. Then, there exists an element $b\in D$ such that $ab\ne ba$ and the division subring $D_1=F(a,b)$ generated by $a$ and $b$ over $F$ is centrally finite. 
\end{lemma}
\begin{proof} Let $D^*$ be the multiplicative group of $D$. Assume that $n$ is the order of $a$. Then, $1, a, a^2,\dots, a^{n-1}$ are distinct roots of the polynomial $t^n-1\in F[t]$. Hence, $F(a)/F$ is a Galois field extension, and so  
	$$|\text{Gal}(F(a)/F)|= [F(a):F]=n.$$
	Since $a\notin F,$ $n\geq 2$, and there exists a non-trivial element $\varphi\in \text{Gal}(F(a)/F)$. Then, $\varphi(a)=a^i$ for some $i>1.$ By the Skolem-Noether Theorem (see~\cite{Bo_Draxl_1983}), there exists $b\in D^*$ such that $bab^{-1}=a^i$. Therefore, $ba=a^ib\ne ab$. It can be proved by induction that  $$ba^s=a^{si}b \text{ for any integer } s\geq 0.\eqno(*)$$ 
	Let $A$ be the $F$-subspace of $D$ generated by the set $\{a^sb^r\mid s, r\in \mathbb{N}\}$. In view of $(*)$, one can easily verify that $A$ is a subring  which is a finite dimensional $F$-subspace of $D$. It implies that $A$ is a division subring of $D$. Clearly, $A=F(a,b)=D_1.$  Since $F\subseteq F_1=Z(D_1)\subseteq D_1$ and $[D_1:F]<\infty$, it follows that $[D_1:F_1]<\infty$.  
\end{proof}
\begin{theorem}\label{th:3.12} Let $D$ be a division ring with center $F.$ Assume that $G$ is an almost subnormal subgroup of $D^*.$ If $G$ is radical over $F$, then $T(G)$ is a subgroup of $F^*$. Moreover, $T(G)=\mathit{Tor}(G),$ and $G$ is central-by-(torsion-free).
\end{theorem}
\begin{proof} In view of Proposition~\ref{pro:3.5}, it suffices to show that $T(G)$ is contained in $F.$ Suppose by contrary that there is an element $a\in T(G)\backslash F$. By Lemma~\ref{lem:3.11}, there exists a division subring $D_1$ of $D$ such that $a\in D_1\backslash F_1$ and $[D_1:F_1]$ is finite, where $F_1$ is the center of $D_1$. Observe that $G_1=G\cap D_1^*$ is a non-central almost subnormal subgroup of $D_1^*$ containing $a.$ Let $G_2=[G_1,G_1]$ be the commutator subgroup of $G_1.$ For any element $d\in G_2,$ there is an integer $n(d)$ such that  $d^{n(d)}=\alpha\in F\subseteq F_1.$ By~\cite[Sublemma]{Pa_Herstein_1978}, $\alpha$ is a root of unity, that is, $\alpha^m=1$ for some $m\geq 1$, and so, the order of $d$ is finite. Therefore, $G_2$ is a periodic almost subnormal subgroup of $D_1.$ In view of Theorem~\ref{th:3.2}, $G_2\subseteq F_1$. By Lemma~\ref{lem:3.10}, there exists an element $g_1\in G_1$ such that $[g_1,a]$ is not in $F_1.$ On the other hand, $[g_1,a]\in [G_1,G_1]= G_2\subseteq F_1$, a contradiction. We conclude that $T(G)\leq F^*.$ It is obvious that $T(G)=\mathit{Tor}(G),$ and $T(G)\subseteq F^*\cap G=Z(G)$ by Lemma~\ref{lem:3.9}. Thus, $G$ is central-by-(torsion-free). 
\end{proof}
\begin{theorem}\label{th:3.13} Let $D$ be a division ring, and $G$  an almost subnormal subgroup of the multiplicative group $D^*.$ If $G$ is (locally generalized radical)-by-periodic, then $G$ is central-by-(torsion-free).
\end{theorem}
\begin{proof} Assume that $G$ is (locally generalized radical)-by-periodic. Then, $G$ contains a locally generalized radical normal subgroup $N$ of $G$ such that the quotient group $G/N$ is periodic. Observe that $N$ is also an almost subnormal subgroup of $D^*.$ Hence, in view of Theorem~\ref{th:2.5}, $N\leq F^*,$ and so $N\leq F^*\cap G=Z(G)$ by Lemma~\ref{lem:3.9}. Since $G/N$ is periodic, so is the central factor group $G/Z(G).$ Now, for all $a\in G,$ one has $aZ(G)\in G/Z(G)$ is periodic, there exists a positive integer $k$ such that $a^k\in Z(G)=F^*\cap G\leq F^*$ by Lemma~\ref{lem:3.9}. This shows that $G$ is radical over $F.$ Thus, according to Theorem~\ref{th:3.12}, $G$ is central-by-(torsion-free).
\end{proof}
\begin{theorem}\label{th:3.14} Let $D$ be a division ring, and $G$ an almost subnormal subgroup of the multiplicative group $D^*.$ If $G$ is a periodic-by-(locally generalized radical) almost subnormal in $D^*$, then $G$ is central.
\end{theorem}
\begin{proof} Assume that $G$ is periodic-by-(locally generalized radical). Then, $G$ contains a periodic normal subgroup $N$ of $G$ such that $G/N$ is locally generalized radical. Observe that $N$ is also an almost subnormal subgroup of the group $D^*.$ Hence, according to Theorem~\ref{th:3.2}, $N\subseteq F^*,$ and so $N\subseteq F^*\cap G=Z(G).$ On the other hand, since $G/N$ is locally generalized radical, so is $G/Z(G).$ It is evident that $G$ is locally generalized radical. Thus, in view of Theorem~\ref{th:2.5}, $G$ is central.
\end{proof}
\begin{theorem}\label{th:3.15} Let $D$ be a division ring whose center $F$ is uncountable. Assume that $G$ is an almost subnormal subgroup of the multiplicalive group $D^*.$ If $G$ is radical over $F,$ then $G$ is central.
\end{theorem}
\begin{proof} Suppose by contrary that $G$ is non-central. According to~\cite[Theorem 2.1]{Pa_ChNaHa_24}, $G$ contains a normal subgroup $N$ of finite index such that $N$ is a non-central subnormal subgroup of the multiplicative group $D^*.$ Since $G$ is radical over $F,$ it follows that the subgroup $N$ is also radical over $F.$ Hence, in view of~\cite[Theorem 2]{Pa_He_80}, $N$ is central, a contradiction. Consequently, $G$ must be central. 
\end{proof}
\section{Periodic ascendant skew linear subgroups}\label{S4}

In this section, we continue to study the structure of periodic skew linear groups we have began in the previous section but in more general circumstance. Namely, we consider periodic ascendant subgroups of the multiplicative group $D^*$ of a division ring $D$. Firstly, we give the affirmative answer to \cite[Question 1]{Pa_ChuaHai_2026} by Theorem~ \ref{th:4.1}. This theorem will be used to prove that every periodic-by-(generalized radical) ascendant subgroup of $D^*$ must be central (see Theorem~\ref{th:4.7}). Further, we prove that every periodic locally generalized radical ascendant subgroup of the group $\mathrm{GL}_n(D)$ is central provided that $D$ is not a locally finite field in case $n\ge 2$ (see Theorem~\ref{th:4.9}). Also, we give the positive answer to Conjecture \ref{conj:2} for a weakly locally finite division ring $D$ (see Theorem \ref{th:4.4}); for a division ring $D$ whose center $F$ is a locally finite field (see Theorem \ref{th:4.5}). Finally, we prove  that if an ascendant subgroup $G$ of $D^*$ is radical over $F$, then every periodic element of $G$ must be contained in $F$ (see Theorem~\ref{th:4.3}). This theorem gives some useful information to verify Conjecture~\ref{conj:2}. 

\begin{theorem}\label{th:4.1} Let $D$ be a division ring, and  $G$ an ascendant subgroup of the multiplicative group $D^*$. If $G$ is periodic, then $G$ is central.
\end{theorem} 
\begin{proof}
	Suppose that $G$ is a periodic ascendant subgroup of the multiplicative group $D^*$. We have to prove that $G\le F^*$. There are two the following cases to examine.
	
	\medskip 
	
	\textit{Case 1. $D$ is centrally finite:}
	
	\medskip
	Since $D$ is centrally finite, the multiplicative group $D^*$ of the division ring $D$ can be considered as a subgroup of the general linear group $\mathrm{GL}_n(F),$ where $F$ is the center of $D$ and $[D:F]=n.$ By Tits' Alternative~\cite{Pa_Tits_1972}, either $G$ is solvable-by-(locally finite) or $G$ contains a non-cyclic free subgroup. The second case cannot occur because $G$ is a periodic group. Therefore, $G$ is a solvable-by-(locally finite) group. Then, $G$ contains a normal solvable subgroup, say $N$, such that $G/N$ is locally finite. Being a subgroup of a periodic group, $N$ is a periodic solvable group. Then, according to~\cite[1.3.5]{Bo_Lennox_2004}, $N$ is a locally finite group. Being an extension of a locally finite group by a locally finite group, $G$ is a locally finite group. Hence, we conclude that $G$ is a locally finite ascendant subgroup of the multiplicative group $D^*.$ Then, according to~\cite[Proposition 1.2.15]{Bo_Dixon_2017}, the normal closure $G^{D^*}$ of $G$ in $D^*$ is a locally finite normal subgroup of $D^*.$ By Theorem~\ref{th:3.2}, it follows that $G^{D^*}$ is central, and so the group $G$ is central, too.
	
	\medskip
	
	\textit{Case 2. $D$ is not centrally finite:}
	
	\medskip
	
	Suppose by contrary that $G$ is non-central, and assume that $a\in G\backslash F$. In view of Lemma~\ref{lem:3.11}, there exists an element $b\in D$ such that $ab\ne ba$ and $D_1=F(a,b)$ is a centrally finite division ring. Let $F_1$ be the center of $D_1$. Observe that $G_1=G\cap D_1$ is a periodic ascendant subgroup of $D_1^*$, and so by \textit{Case 1}, we have $G_1\le F_1^*$. This implies that $ab=ba$, a contradiction. Thus, we conclude that $G$ must be central.
	
	The proof of the theorem is now complete. 
\end{proof} 

\begin{theorem}\label{th4.2} Let $D$ be a division ring, and $G$ an ascendant subgroup of the multiplicative group $D^*$. Then, $T(G)$ is a subgroup of $G$ if and only if $T(G)$ is central. In case this, $T(G)=\mathit{Tor}(G)$ and the group $G/\mathit{Tor}(G)$ is torsion-free.
\end{theorem}
\begin{proof}
If $T(G)$ is a subgroup of $G,$ then clearly $T(G)$ is normal in $G$, and so $T(G)$ is ascendant in $D^*$. By Theorem~\ref{th:4.1}, $T(G)$ is central. The converse is obvious.
\end{proof}

Let $G$ be a subnormal subgroup of the multiplicative group $D^*$ of a division ring $D$. In \cite[Theorem 9]{Pa_Herstein_1978}, Herstein proved that if $G$ is radical over the center $F$ of $D$ then every periodic element of $G$ must be contained in $F$. The next theorem shows that this property remains true in case $G$ is any ascendant subgroup of $D^*$.

\begin{theorem}\label{th:4.3} Let $D$ be a division ring with the center $F.$ Assume that $G$ is an ascendant subgroup of the multiplicative group $D^*$ of $D$. If $G$ is radical over $F$, then every periodic element of $G$ must be in $F$. Consequently, $T(G)=Tor(G)$ is a normal abelian subgroup of $G$ and $G/Tor(G)$ is a torsion-free group.
\end{theorem}
\begin{proof} Suppose that $a\in G$ is a periodic element. If $a\notin F$, then, according to Lemma~\ref{lem:3.11}, there exists an element $b\in D$ such that $ab\ne ba$ and the division subring $D_1=F(a,b)$ generated by $a$ and $b$ over $F$ is centrally finite with the center, say, $F_1$. It is clear that $G_1=G\cap D_1$ is an ascendant subgroup of $D_1^*$ and $G_1$ is radical over $F_1.$ Let $H_1=[G_1,G_1]$ be the derived subgroup of $G_1.$ If $x\in H_1,$ then $x^{n(x)}=\alpha\in F_1.$ Thus, by~\cite[Sublemma]{Pa_Herstein_1978}, $\alpha$ is a root of unity. This implies that $x$ is a periodic element. Hence, we conclude that $H_1$ is a periodic ascendant subgroup of $D_1^*.$ In view of Theorem~\ref{th:4.1}, $H_1$ is contained in $F_1.$ Thus, the subgroup $G_1$ is solvable ascendant in $D_1^*$. Now, according to~\cite[Theorem 3.8]{Pa_ChuaHai_2026}, it follows that $G_1$ is contained in $F_1,$ and so $ab=ba,$ a contradiction.
\end{proof}	

The following result shows that Conjecture~\ref{conj:2}  holds for weakly locally finite division rings.
\begin{theorem}\label{th:4.4} Let $D$ be a weakly locally finite division ring with center $F.$ Assume that $G$ is an ascendant subgroup of $D^*.$ If $G$ is radical over $F,$ then $G$ is central.
\end{theorem}
\begin{proof} If $a\in G',$ then there exists some positive integer $n(a)$ such that $a^{n(a)}\in F^*.$ Hence, $a^{n(a)}\in F^*\cap D'=Z(D').$ By~\cite[Lemma 7]{Pa_dbh_2019}, it follows that $a$ is a periodic element. We conclude that $G'$ is periodic. Observe that $G'$ is ascendant in $D^*.$ By Theorem~\ref{th:4.1}, $G'$ is central, which implies that $G$ is solvable. Hence, according to~\cite[Theorem 3.8]{Pa_ChuaHai_2026}, it follows that $G$ is central. 
\end{proof}
The following result gives the affirmative answer to Conjecture~\ref{conj:2} in case $D$ is a division ring whose center is a locally finite field.
\begin{theorem}\label{th:4.5} Let $D$ be a division ring with center $F$, and $G$ an ascendant subgroup of $D^*.$ If $G$ is radical over $F,$ then $G$ is central provided that $F$ is a locally finite field.
\end{theorem}
\begin{proof} Consider an arbitrary element $a\in G$. Then, there exists some positive integer $n(a)$ such that $a^{n(a)}\in F.$ Since $F$ is a locally finite field, $a^{n(a)}$ belongs to some finite subfield $K$ of $F,$ which implies that $a$ is a periodic element. We conclude that $G$ is a periodic group, and so, $G$ is central by Theorem~\ref{th:4.1}.
\end{proof}
\begin{theorem}\label{th:4.6} Let $D$ be a division ring with center $F.$ Assume that $G$ is an ascendant subgroup of the multiplicative group $D^*.$ If $G$ is (generalized radical)-by-periodic, then $G$ is central-by-(torsion-free).
\end{theorem}
\begin{proof} Assume that $G$ is (generalized radical)-by-periodic. Then, $G$ contains a generalized radical normal subgroup $N$ of $G$ such that $G/N$ is periodic. Observe that $N$ is also a generalized radical ascendant subgroup of the multiplicative group $D^*.$ Thus, in view of~\cite[Theorem 3.8]{Pa_HaiChua_2025}, $N$ is central, and so $N\subseteq G\cap F^*=Z(G).$ This implies that the central factor group $G/Z(G)$ is periodic, and so $G$ is radical over $F.$ Thus, according to Theorem~\ref{th:4.3}, $G$ is central-by-(torsion-free).
\end{proof}

\begin{theorem}\label{th:4.7} Let $D$ be a division ring, and $G$  an ascendant subgroup of the multiplicative group $D^*.$ If $G$ is periodic-by-(generalized radical), then $G$ is central.
\end{theorem}
\begin{proof} Assume that $G$ is periodic-by-(generalized radical) ascendant in $D^*.$ Then, $G$ contains a periodic normal subgroup $H$ of $G$ such that $G/H$ is generalized radical. Observe that $H$ is a periodic ascendant subgroup of the multiplicative group $D^*.$ According to Theorem~\ref{th:4.1}, $H$ is central, and so by \cite[Lemma 2.1]{Pa_ChuaHai_2026}, $G$ is generalized radical ascendant in $D^*$. Hence, in view of~\cite[Theorem 3.8]{Pa_ChuaHai_2026}, $G$ is central.
\end{proof}
To prove the next result, we need the following result in~\cite[Lemma 1.2.17]{Bo_Dixon_2017}.
\begin{proposition}\label{pro:4.8}{\rm \cite[Lemma 1.2.17]{Bo_Dixon_2017}} If $G$ is a periodic locally generalized radical group, then $G$ is locally finite.
\end{proposition}
\begin{theorem}\label{th:4.9} Let $D$ be a division ring, and $n$ a positive integer. Assume that $G$ is a subgroup of the general skew linear group $\mathrm{GL}_n(D).$ Suppose further that  $D$ is not a locally finite field in case $n\geq 2.$ Then, the following assertions hold.
	\begin{itemize}
		\item [(1)]  If $G$ is periodic locally generalized radical ascendant subgroup of the general skew linear group $\mathrm{GL}_n(D),$ then $G$ is central.
		\item [(2)] If $G$ is periodic locally generalized radical ascendant subgroup of the special skew linear group $\mathrm{SL}_n(D),$ then $G$ is contained in $Z(\mathrm{SL}_n(D)).$
		\item [(3)] The projective general skew linear group $\mathrm{PGL}_n(D)$ contains no non-trivial periodic locally generalized radical ascendant subgroups.
		\item [(4)] The  projective special skew linear group $\mathrm{PSL}_n(D)$ contains no non-trivial periodic locally generalized radical ascendant subgroups.
	\end{itemize}
\end{theorem}
\begin{proof} (1) According to Proposition~\ref{pro:4.8}, $G$ is a locally finite ascendant subgroup of $\mathrm{GL}_n(D)$. Hence, in view of~\cite[Proposition 1.2.15]{Bo_Dixon_2017}, the normal closure $G^{\mathrm{GL}_n(D)}$ is locally finite normal in $\mathrm{GL}_n(D).$ By Theorem~\ref{th:3.2}, $G^{\mathrm{GL}_n(D)}$ is central, so does $G.$
	
(2) Observe that $G$ is periodic locally generalized radical ascendant subgroup of $\mathrm{GL}_n(D).$ By the conclusion of (1), $G\subseteq Z(\mathrm{GL}_n(D)).$ Hence, $$G\subseteq Z(\mathrm{GL}_n(D))\cap \mathrm{SL}_n(D)=Z(\mathrm{SL}_n(D)).$$
	
(3) Suppose by contrary that $\mathrm{PGL}_n(D)$ contains a non-trivial periodic locally generalized radical ascendant subgroup $U$. Then, in view of Proposition~\ref{pro:4.8}, it follows that $U$ is a locally finite ascendant subgroup of $\mathrm{PGL}_n(D)$. Hence, in view of~\cite[Proposition 1.2.15]{Bo_Dixon_2017}, the normal closure $U^{\mathrm{PGL}_n(D)}$ of $U$ in $\mathrm{PGL}_n(D)$ is a non-trivial locally finite normal subgroup of $\mathrm{PGL}_n(D).$ By Theorem~\ref{th:3.4}, it follows that $n=1,$ and that $U^{D^*/F^*}$ is a non-trivial locally finite normal subgroup of $D^*/F^*.$ We can write $U^{D^*/F^*}=V/F^*,$ where $V$ is a non-central normal subgroup of $D^*.$ Then, $V$ is abelian-by-(locally finite), that contradicts to~\cite[Theorem 2.6]{Pa_ChNaHa_24}. 
	
(4) Suppose by contrary that $\mathrm{PSL}_n(D)$ contains a non-trivial periodic locally generalized radical ascendant subgroup. Since $\mathrm{PSL}_n(D)$ is isomorphic to the derived subgroup $\mathrm{PGL}_n'(D)$ of $\mathrm{PGL}_n(D),$ it follows that $\mathrm{PGL}_n'(D)$ contains a non-trivial periodic locally generalized radical ascendant subgroup and so does the group $\mathrm{PGL}_n(D).$ This contradicts to the conclusion of (3), and so the result follows. 
	
	The proof of the theorem is now complete.
\end{proof}
\begin{corollary}\label{cor:4.10} Let $D$ be a division ring, and $n$ a positive integer. Assume that $G$ is a subgroup of the general skew linear group $\mathrm{GL}_n(D).$ Suppose further that  $D$ is not a locally finite field in case $n\geq 2.$ Then, the following assertions hold:
	\begin{itemize}
		\item [(1)] If $G$ is a locally generalized radical ascendant subgroup of $\mathrm{GL}_n(D),$ then $T(G)$ is a subgroup of $G$ if and only if $T(G)\subseteq Z(\mathrm{GL}_n(D)).$ In case this, $T(G)=\mathit{Tor}(G)$, and the quotient group $G/\mathit{Tor}(G)$ is torsion-free.	
		\item [(2)] If $H$ is a locally generalized radical ascendant subgroup of $\mathrm{SL}_n(D),$ then $T(H)$ is a subgroup of $H$ if and only if $T(H)\subseteq Z(\mathrm{SL}_n(D)).$ In case this, $T(H)=\mathit{Tor}(H)$, and the quotient group $H/\mathit{Tor}(H)$ is torsion-free.
		\item [(3)] If $U$ is a locally generalized radical ascendant subgroup of the general projective skew linear group $\mathrm{PGL}_n(D),$ then $T(U)$ is a subgroup of $U$ if and only if $T(U)=1.$ In case this, $U$ is a torsion-free group.
		\item [(4)] If $V$ is a locally generalized radical ascendant subgroup of the  special projective skew linear group $\mathrm{PSL}_n(D),$ then $T(V)$ is a subgroup of $V$ if and only if $T(V)=1.$ In case this, $V$ is a torsion-free group.
	\end{itemize}
\end{corollary}
\begin{proof} The result immediately follows from Theorem~\ref{th:4.9}. 
\end{proof}

\section{Locally generalized radical linear groups\\ over left artinian rings}\label{S5}

In this section, we investigate the stucture of locally generalized radical linear groups over a left artinian ring. The main results here are Theorems \ref{th:5.7} and \ref{th:5.8}, which give in particular the affirmative answer to GBP for almost subnormal subgroups of the general linear group $\mathrm{GL}_n(R)$ over a left artinian ring $R.$

Let $R$ be a ring with the Jacobson radical $J=J(R)$, and $I$ an ideal of $R$. we consider the subset $G(I)$ of $R$, which is defined as follows
$$G(I):=1+I=\{1+a\mid a\in I\}.$$
If $I$ is a nilpotent ideal of $R,$ then the \textit{nilpotency class} of $I$ is the minimal positive integer $k$ such that $I^k=0$. Clearly, every nilpotent ideal is a nil ideal of $R$. Recall that the converse also true if $R$ is a left artinian ring. The following result shows that if $I$ is a nilpotent ideal of $R,$ then $G(I)$ is a normal nilpotent subgroup of the unit group $R^\times$ and $R^\times/G(I)\cong (R/I)^\times,$ and that $G(I)\trianglelefteq G(J).$

\begin{lemma}\label{lem:5.1} Let $R$ be a ring with its Jacobson radical $J=J(R)$, $R^\times$ the unit group of $R$, and $I$ an ideal of $R$. Then, the following assertions hold:
\begin{itemize}
\item [(1)] $G(I)$ is a submonoid of $R^\times.$
\item [(2)] If $I\subseteq J,$ then $G(I)$ is a normal subgroup of $R^\times$ and $R^\times/G(I)\cong (R/I)^\times.$ In particular, $G(J)$ is normal in $R^\times$ and $R^\times/G(J)\cong (R/J)^\times.$ Moreover, the unit group $R^\times$ has a normal series $$1\trianglelefteq G(I)\trianglelefteq G(J)\trianglelefteq R^\times.$$
\item [(3)] If $I$ is a nilpotent ideal of nilpotency class $k$, then $G(I)$ is a nilpotent group of nilpotency class at most $k.$ Moreover, $G(I)\trianglelefteq R^\times$ and $$R^\times/G(I)\cong (R/I)^\times.$$
\end{itemize}	
\end{lemma} 
\begin{proof} (1) For every $x=1+a, y=1+b\in G(I)$, where $a,b\in I$, we have
 $$xy=(1+a)(1+b)=1+(a+b+ab)\in G(I),$$	
which shows that $G(I)$ is closed under multiplication, and the identity element $1\in G(I).$ Consequently, $G(I)$ is a submonoid of the unit group $R^\times.$

(2) Suppose that $I\subseteq J$. If $a\in I,$ then $a\in J,$ and so $1+a\in R^\times.$ Thus, $G(I)\subseteq R^\times.$ Hence, by conclusion of (1), we get that $G(I)$ is a subgroup of $R^\times$. Further, for any $r\in R^\times$ and $a\in I$, we have
$$rxr^{-1}=r(1+a)r^{-1}=1+rar^{-1}\in G(I),$$ which shows that $G(I)$ is normal in $R^\times.$ Since $I\subseteq J,$ it follows that $G(I)\subseteq G(J).$ As a consequence, we get the normal series 
	$$1\trianglelefteq G(I)\trianglelefteq G(J)\trianglelefteq R^\times.$$
Now, consider any non-zero element $a\in R$. Observe that $a\in R^\times$ if and only if $\overline{a}=a+I\in (R/I)^\times$ because $I\subseteq J$. It follows that the natural homomorphism $R^\times\longrightarrow (R/I)^\times$ given by $a\mapsto \overline{a},$ is a surjective homomorphism whose kernel is $G(I),$ and hence $R^\times/G(I)\cong (R/I)^\times.$
		
(3) Suppose that $I$ is a nilpotent ideal of nilpotency $k$. For all $r\in \{1,2,\ldots,k\},$ let $$G_r(I)=1+I^r\subseteq G(I).$$
Since $I^r\subseteq J$, in view of (2), $G_r(I)$ is a normal subgroup of $R^\times$, and we have the following  normal series of subgroups
$$1=G_k(I)\le G_{k-1}(I)\le\cdots\le G_1(I)=G(I).$$
We shall prove that this is a central series of $G(I)$. Indeed, for each $1\leq s, r\leq k,$ let $x=1+a\in G_s(I),$ with $a\in I^s$ and $y=1+b\in G_r(I),$ with $b\in I^r.$ We have
	\begin{gather}
		\begin{split}
			[x,y]-1&=x^{-1}y^{-1}xy-1\\
			&=x^{-1}y^{-1}(xy-yx)\\
			&=x^{-1}y^{-1}(ab-ba)\\
			&\in R\cdot(I^s\cdot I^r+I^r\cdot I^s)\subseteq I^{s+r}.
		\end{split}\notag
	\end{gather}
	This implies that $[x,y]\in G_{s+r}(I),$ and so $[G_s(I),G_r(I)]\le G_{s+r}(I).$ Hence, we conclude that $G(I)$ is a nilpotent group of nilpotency class at most $k$, and at this point,  the proof of the lemma is  complete.
\end{proof}
Let $I$ be an ideal of a ring $R$, and $n$ a positive integer. It is known that $\mathrm{M}_n(I)$ is an ideal of the matrix ring $\mathrm{M}_n(R)$ and $\mathrm{M}_n(R)/\mathrm{M}_n(I)\cong\mathrm{M}_n(R/I).$ If $I$ is nilpotent in $R$, then $\mathrm{M}_n(I)$ is nilpotent in $\mathrm{M}_n(R).$ According to Lemma~\ref{lem:5.1}, $$G(\mathrm{M}_n(I))=1+\mathrm{M}_n(I)$$ is a  nilpotent normal subgroup of the general linear group $\mathrm{GL}_n(R)$ and $$\mathrm{GL}_n(R)/G(\mathrm{M}_n(I))\cong (\mathrm{M}_n(R)/\mathrm{M}_n(I))^\times\cong \mathrm{GL}_n(R/I).$$ Let $J=J(R)$ be the Jacobson radical of $R.$ Since $I$ is nilpotent, it follows that $I\subseteq J,$ and so $\mathrm{M}_n(I)\subseteq \mathrm{M}_n(J).$ This implies that $G(\mathrm{M}_n(I))\subseteq G(\mathrm{M}_n(J)).$ It is known that $\mathrm{M}_n(J)=J(\mathrm{M}_n(R))$ is the Jacobson radical of the matrix ring $\mathrm{M}_n(R).$ Thus, in view of Lemma~\ref{lem:5.1}, $G(\mathrm{M}_n(J))$ is a normal subgroup of $\mathrm{GL}_n(R)$ and $$\mathrm{GL}_n(R)/G(\mathrm{M}_n(J))\cong (\mathrm{M}_n(R)/\mathrm{M}_n(J))^\times\cong \mathrm{GL}_n(R/J).$$ Now, we set $\mathrm{JL}_n(R):=G(\mathrm{M}_n(J)).$ Then,  $$\mathrm{JL}_n(R)=1+\mathrm{M}_n(J)=1+J(\mathrm{M}_n(R))$$ is a normal subgroup of the group $\mathrm{GL}_n(R)$ and $\mathrm{GL}_n(R)/\mathrm{JL}_n(R)\cong \mathrm{GL}_n(R/J).$ Moreover, the general linear group $\mathrm{GL}_n(R)$ has a normal series $$1\trianglelefteq G(\mathrm{M}_n(I))\trianglelefteq \mathrm{JL}_n(R)\trianglelefteq \mathrm{GL}_n(R),$$ where $G(\mathrm{M}_n(I))$ is a  nilpotent normal subgroup of the group $\mathrm{GL}_n(R)$ and $$\mathrm{GL}_n(R)/G(\mathrm{M}_n(I))\cong \mathrm{GL}_n(R/I).$$ 
We summary the facts above in the following lemma.
\begin{lemma}\label{lem:5.2} Let $R$ be a ring, $J=J(R)$ its Jacobson radical, and $n$ is a positive integer. If $I$ is a nilpotent ideal of $R,$ then $\mathrm{M}_n(I)$ is a nilpotent ideal of the matrix ring $\mathrm{M}_n(R),$  and  $\mathrm{M}_n(I)\subseteq \mathrm{M}_n(J).$ Moreover, the following assertions hold:
\begin{itemize}
\item [(1)] $G(\mathrm{M}_n(I))$ is a normal nilpotent subgroup of $\mathrm{GL}_n(R)$ and $$\mathrm{GL}_n(R)/G(\mathrm{M}_n(I))\cong \mathrm{GL}_n(R/I).$$
\item [(2)] $\mathrm{JL}_n(R)\trianglelefteq\mathrm{GL}_n(R)$ and $\mathrm{GL}_n(R)/\mathrm{JL}_n(R)\cong \mathrm{GL}_n(R/J).$
\item [(3)] The general linear group $\mathrm{GL}_n(R)$ has a normal series $$1\trianglelefteq G(\mathrm{M}_n(I))\trianglelefteq \mathrm{JL}_n(R)\trianglelefteq \mathrm{GL}_n(R).$$
\item [(4)] If $J$ is a nilpotent ideal of $R,$ then $\mathrm{JL}_n(R)\triangleleft\mathrm{GL}_n(R),$ and $$\mathrm{GL}_n(R)/\mathrm{JL}_n(R)\cong\mathrm{GL}_n(R/J).$$
\end{itemize}
\end{lemma}
The normal subgroup $\mathrm{JL}_n(R)$ of the general linear group $\mathrm{GL}_n(R)$ in Lemma~\ref{lem:5.2} is said to be the \textit{Jacobson linear group} of degree $n$ over a ring $R.$
\begin{lemma}\label{lem:5.3} Let $R$ be a ring, $J=J(R)$ its Jacobson radical, and $n$ a positive integer. Then, the following assertions hold:
\begin{itemize}
\item [(1)] If $R$ has characteristic zero, then $\mathrm{JL}_n(R)$ is torsion-free nilpotent.
\item [(2)] If $R$ has positive characteristic $p$, then $\mathrm{JL}_n(R)$ is locally finite. Furthermore, the group $\mathrm{GL}_n(R)$ is periodic if and only if the group $\mathrm{GL}_n(R/J)$ is periodic.
\end{itemize} 
\end{lemma} 
\begin{proof} (1) Assume that $x\in \mathrm{JL}_n(R)$ is a non-trivial torsion element of order $m\geq 1.$ Write $x=1+a,$ where $0\ne a\in \mathrm{M}_n(J)$. Then, we have $$1=x^m=(1+a)^m=1+ma+\sum_{i=2}^{m}\mathrm{C}_m^ia^i.$$ Pick $b=\displaystyle\frac{1}{n}\sum_{i=2}^{m}\mathrm{C}_m^ia^{i-1}\in \mathrm{M}_n(J)$, and therefore $1+b\in \mathrm{JL}_n(R).$ Observe that $$ma(1+b)=ma+\sum_{i=2}^{m}\mathrm{C}_m^ia^i=0.$$ Since $1+b\in \mathrm{JL}_n(R)$, it follows that $ma=0.$ We also observe that the matrix ring $\mathrm{M}_n(R)$ has characteristic zero. This shows that $a=0$, a contradiction. Thus, $\mathrm{JL}_n(R)$ is torsion-free. Hence, $\mathrm{JL}_n(R)$ is a torsion-free nilpotent group.  
	
(2) Consider an arbitrary element $x=1+a$ ($a\in \mathrm{M}_n(J)$) in the group $\mathrm{JL}_n(R)$. Clearly, we can choose some positive integer $s$ such that $a^{p^s}=0$. Indeed, assume that $m$ is a positive integer such that $(\mathrm{M}_n(J))^m=0$. To get $a^{p^s}=0$, it suffices to take any integer $s$ such that $p^s>m$. With such an integer $s$, we have  $$x^{p^s}=(1+a)^{p^s}=1+a^{p^s}=1,$$ which implies that $x$ is a periodic element. Hence, we conclude that $\mathrm{JL}_n(R)$ is a periodic nilpotent group, and so, it follows from Proposition~\ref{pro:4.8} that $\mathrm{JL}_n(R)$ is locally finite. Now, suppose that the general linear group $\mathrm{GL}_n(R)$ is periodic. Since $$\mathrm{GL}_n(R)/\mathrm{JL}_n(R)\cong \mathrm{GL}_n(R/J),$$ it follows that $\mathrm{GL}_n(R/J)$ is periodic. Conversely, assume that $\mathrm{GL}_n(R/J)$ is a periodic group. Then, $\mathrm{GL}_n(R)/\mathrm{JL}_n(R)$ is periodic. Consider an arbitrary element $z\in \mathrm{GL}_n(R)$. Then, $z\mathrm{JL}_n(R)\in \mathrm{GL}_n(R)/\mathrm{JL}_n(R)$ is a periodic element, say of order $k\ge 1$, and we have $z^k\mathrm{JL}_n(R)=(z\mathrm{JL}_n(R))^k=\mathrm{JL}_n(R)$, which implies that $z^k\in \mathrm{JL}_n(R)$. Since $\mathrm{JL}_n(R)$ is locally finite, it follows that $z^k$ is  periodic, and so $z$ is periodic, too. Hence, the group $\mathrm{GL}_n(R)$ is periodic, and the proof of the lemma is now complete.	
\end{proof}	
\begin{lemma}\label{lem:5.4} Let $R$ be a left artinian ring, $J=J(R)$ its Jacobson radical, and $n$ a positive integer. Then, there exist positive integers $n_1, n_2,\ldots, n_r$ and division rings $D_1, D_2,\ldots, D_r$ such that  $$\mathrm{M}_n(R/J)\cong \prod_{i=1}^{r}\mathrm{M}_{n_i}(D_i).$$ Moreover, $\mathrm{JL}_n(R)$ is a nilpotent normal subgroup of the group $\mathrm{GL}_n(R),$ and $$\mathrm{GL}_n(R)/\mathrm{JL}_n(R)\cong\mathrm{GL}_n(R/J)\cong \prod_{i=1}^{r}\mathrm{GL}_{n_i}(D_i).$$
\end{lemma}
\begin{proof} Since $R$ is a left artinian ring, $J=J(R)$ is a nilpotent ideal of $R$, and the factor ring $R/J$ is semisimple. By the Wedderburn-Artin Theorem, there exist positive integers $t_1,t_2,\ldots, t_r$ and division rings $D_1, D_2,\ldots, D_r$ such that $$R/J\cong\prod_{i=1}^{r}\mathrm{M}_{t_i}(D_i).$$ Clearly, $$\mathrm{M}_n(R/J)\cong\prod_{i=1}^{r}\mathrm{M}_n(\mathrm{M}_{t_i}(D_i))\cong\prod_{i=1}^{r}\mathrm{M}_{nt_i}(D_i)=\prod_{i=1}^{r}\mathrm{M}_{n_i}(D_i),$$ where $n_i=nt_i$. It follows that $\mathrm{GL}_n(R/J)\cong \prod_{i=1}^{r}\mathrm{GL}_{n_i}(D_i).$ On the other hand, according to Lemma~\ref{lem:5.2}, $\mathrm{JL}_n(R)$ is a nilpotent normal subgroup of the general linear group $\mathrm{GL}_n(R),$ and $\mathrm{GL}_n(R)/\mathrm{JL}_n(R)\cong\mathrm{GL}_n(R/J).$ The lemma is proved.
\end{proof}
To proceed further, we need some obvious results on periodic groups, and on locally finite groups.  For the convenience of use, let us record them in the following proposition.
\begin{proposition}\label{pro:5.5} Let $G_1, G_2,\ldots, G_r$ be groups and $G=G_1\times G_2\times\cdots\times G_r.$ Then, the following statements hold:
	\begin{itemize}
		\item [(1)] $G$ is periodic if and only if all $G_i$ are periodic.
		\item [(2)] $G$ is locally finite if and only if all $G_i$ are locally finite.	
	\end{itemize}	
\end{proposition}
\begin{proposition}\label{pro:5.6} Let $R$ be a left artinian ring, $J=J(R)$ its Jacobson radical,  and $n$ a positive integer. If the general linear group $\mathrm{GL}_n(R)$ is periodic, then $\mathrm{GL}_n(R)$ is locally finite. Moreover, there exist locally finite fields $F_1, F_2, \ldots, F_r,$ and positive integers $n_1, n_2,\ldots, n_r$ such that $$\mathrm{M}_n(R/J)\cong \mathrm{M}_{n_1}(F_1)\times \mathrm{M}_{n_2}(F_2)\times\cdots\times \mathrm{M}_{n_r}(F_r).$$
\end{proposition}
\begin{proof} By Lemma~\ref{lem:5.4}, there exist positive integers $n_1, n_2,\ldots, n_r$ and division rings $D_1, D_2,\ldots, D_r$ such that $$\mathrm{M}_n(R/J)\cong \prod_{i=1}^{r}\mathrm{M}_{n_i}(D_i).$$
Moreover, $\mathrm{JL}_n(R)$ is a nilpotent normal subgroup of the group $\mathrm{GL}_n(R),$ and $$\mathrm{GL}_n(R)/\mathrm{JL}_n(R)\cong \prod_{i=1}^{r}\mathrm{GL}_{n_i}(D_i).$$ We can identify the group $\mathrm{GL}_n(R)/\mathrm{JL}_n(R)$ with $\prod_{i=1}^{r}\mathrm{GL}_{n_i}(D_i),$ and write $$\mathrm{GL}_n(R)/\mathrm{JL}_n(R)=\prod_{i=1}^{r}\mathrm{GL}_{n_i}(D_i).$$
Now, if the group $\mathrm{GL}_n(R)$ is periodic, then so is the quotient group $\mathrm{GL}_n(R)/\mathrm{JL}_n(R).$ Then, according to Proposition~\ref{pro:5.5}, all the general skew linear groups $\mathrm{GL}_{n_i}(D_i)$ is periodic. We shall prove that $D_i$ is a locally finite field. Indeed, If $n_i=1,$ then, by Jacobson's Theorem \cite[Theorem 3.9.5]{Bo_Cohn_1995}, $D_i$ is a field. If $D_i$ has characteristic zero, then $D$ contains the field $\mathbb{Q}$ of rational numbers. Since $D_i^*$ is periodic, the multiplicative group $\mathbb{Q}^*$ is periodic, a contradiction. Hence, $D_i$ is a field of prime characteristic, say, $p_i.$ Let $\mathbb{F}_{p_i}$ be the prime subfield of $D_i$, and $S_i=\{a_{1i},a_{2i},\ldots,a_{ki}\}$ a finite subset of $D_i.$ Then, $\mathbb{F}_{p_i}(a_{1i},a_{2i},\ldots,a_{ki})$ is an algebraic extension of $\mathbb{F}_{p_i}.$ Hence, $\mathbb{F}_{p_i}(a_{1i},a_{2i},\ldots,a_{ki})$ is a finite extension of $\mathbb{F}_{p_i}$; that is, $[\mathbb{F}_{p_i}(a_{1i},a_{2i},\ldots,a_{ki}):\mathbb{F}_{p_i}]$ is finite, which  implies that $\mathbb{F}_{p_i}(a_{1i},a_{2i},\ldots,a_{ki})$ is a finite field. Hence, $D_i$ is a locally finite field. Now, we consider the case when $n_i\geq 2.$ Then, according to Lemma~\ref{lem:3.1}, $D_i=F_i$ must be a locally finite field. Hence, $\mathrm{GL}_n(F_i)$ is locally finite. Again by Proposition~\ref{pro:5.5}, $\mathrm{GL}_n(R)/\mathrm{JL}_n(R)$ is locally finite. On the other hand, observe that $\mathrm{JL}_n(R)$ is periodic nilpotent. Hence, according to Proposition~\ref{pro:4.8}, $\mathrm{JL}_n(R)$ is locally finite. Being an extension of a locally finite group by a locally finite group, the group $\mathrm{GL}_n(R)$ is locally finite. 
\end{proof}
The following result shows in particular that the GBP is solved positively for almost subnormal subgroups of the general linear group over left artinian rings.
\begin{theorem}\label{th:5.7} Let $R$ be a left artinian ring and $n$ a positive integer. Assume that $G$ is an almost subnormal subgroup of the general linear group $\mathrm{GL}_n(R).$ Then, the following assertions hold:
\begin{itemize}
\item [(1)] If $G$ is periodic, then $G$ is locally finite.
\item [(2)] If $G$ is locally generalized radical, then $G$ is (nilpotent-by-abelian)-by-(locally finite). Additionally, if the group $G$ is finitely generated, then $H$ is (nilpotent-by-polycyclic)-by-finite. Consequently, $G$ contains a nilpotent normal subgroup $N$ such that the quotient group $G/N$ can be embedded into the general linear group $\mathrm{GL}_m(\mathbb{Z})$ for some $m\geq 1.$
\end{itemize}	
\end{theorem}
\begin{proof} By Lemma~\ref{lem:5.4}, $\mathrm{JL}_n(R)$ is a nilpotent normal subgroup of $\mathrm{GL}_n(R),$ and $$\mathrm{GL}_n(R)/\mathrm{JL}_n(R)= \prod_{i=1}^{r}\mathrm{GL}_{n_i}(D_i)$$ for suitable positive integers $n_1, n_2,\ldots, n_r$ and division rings $D_1, D_2,\ldots, D_r.$ 
	
	Assume that $G$ is a subgroup of the group $\mathrm{GL}_n(R).$ Then, $$G/(G\cap \mathrm{JL}_n(R))\cong G\cdot\mathrm{JL}_n(R)/\mathrm{JL}_n(R)$$ is a subgroup of $\mathrm{GL}_n(R)/\mathrm{JL}_n(R).$ Pick $N=G\cap \mathrm{JL}_n(R)$ and $L=G\cdot\mathrm{JL}_n(R)/\mathrm{JL}_n(R).$ Then, $N$ is normal nilpotent in $G$ and $G/N\cong L.$ Consider the projections $$\pi_i: \mathrm{GL}_n(R)/\mathrm{JL}_n(R)\longrightarrow \mathrm{GL}_{n_i}(D_i).$$ 
	
(1) Assume that $G$ is periodic almost subnormal in $\mathrm{GL}_n(R).$ Then, $L$ is periodic almost subnormal in $\mathrm{GL}_n(R)/\mathrm{JL}_n(R).$ Being the homomorphic image of a periodic almost subnormal subgroup, $\pi_i(L)$ is a periodic almost subnormal subgroup of the general skew linear group $\mathrm{GL}_{n_i}(D_i).$ We claim that $\pi_i(L)$ is locally finite. Indeed, if either $n_i=1$ or $D_i$ is not a locally finite field in case $n_i\geq 2,$ then, according to Theorem~\ref{th:3.2}, $\pi_i(L)$ is central. Hence, $\pi_i(L)$ is a periodic abelian group, and so $\pi_i(L)$ is locally finite. Now, we consider the case that $n_i\geq 2$ and $D_i$ is a locally finite field. Assume that $M_i$ is a finitely generated subgroup of $\pi_i(L)$ that is generated by matrices $A_1, \dots, A_k\in \mathrm{GL}_n(D_i)$. Denote by $K_i$ the subfield of $D_i$ generated by all the entries of all these matrices over the simple subfield $\mathbb{F}_{p_i}$ of $D_i$, where $p_i$ is the characteristic of $D_i$. Then, $K_i$ is a finite field and $M_i$ is a subgroup of $\mathrm{GL}_{n_i}(K_i)$, which implies that $M_i$ is finite, and so $\pi_i(L)$ is locally finite. Being the finite direct product of locally finite groups, $P=\prod_{i=1}^{r} \pi_i(L)$ is locally finite (see Proposition~\ref{pro:5.5}). Since $L$ is isomorphic to the subgroup of $P,$ it follows that $L$ is locally finite. This implies that $G/N$ is locally finite. On the other hand, we observe that $N$ is periodic nilpotent. By Proposition~\ref{pro:4.8}, $N$ is locally finite. Now, since  $G/N$ and $N$ both are locally finite, it follows that $G$ is locally finite. 
	
(2) If $G$ is locally generalized radical almost subnormal in $\mathrm{GL}_n(R),$ then $\pi_i(L)$ is locally generalized radical almost subnormal in $\mathrm{GL}_{n_i}(D_i).$ If $n_i=1$ or $D_i$ is not a locally finite field in case $n_i\geq 2$, then $\pi_i(L)$ is central by Theorem~\ref{th:2.5} and~\cite[Theorem 3.3]{Pa_ChuaHai_2026}. If $D_i$ is a locally finite field, then clearly that $\pi_i(L)$ is locally finite. Hence, we conclude that $\pi_i(L)$ is abelian-by-(locally finite) for all $i.$ It is easy to see that the direct product $Q=\prod_{i=1}^{r} \pi_i(L)$ is abelian-by-(locally finite). Since $L$ is isomorphic to the subgroup of $Q,$ it follows that $L$ is abelian-by-(locally finite). This implies that the quotient group $G/N$ is abelian-by-(locally finite). Since $N$ is nilpotent, $G$ is (nilpotent-by-abelian)-by-(locally finite). If $G$ is finitely generated, then $G/N$ is finitely generated (abelian-by-finite). Since every finitely generated abelian group is polycyclic, it follows that $G/N$ is polycyclic-by-finite. This implies that $G$ is (nilpotent-by-polycyclic)-by-finite. Finally, according to the Auslander-Swan Theorem (see e.g.,~\cite[3.3.1]{Bo_Lennox_2004}), the quotient group $G/N$ can be embedded into the general linear group $\mathrm{GL}_n(\mathbb{Z})$ for some $m\geq 1.$ The theorem is proved.
\end{proof}
\begin{theorem}\label{th:5.8} Let $R$ be a left artinian ring, and $n$ a positive integer. Assume that $G$ is an ascendant subgroup of the general linear group $\mathrm{GL}_n(R).$ If $G$ is generalized radical, then $G$ is (nilpotent-by-abelian)-by-(locally finite). Additionally, if $G$ is finitely generated, then $G$ is (nilpotent-by-polycyclic)-by-finite. Consequently, $G$ contains a nilpotent normal subgroup $N$ such that the quotient group $G/N$ can be embedded into the general linear group $\mathrm{GL}_m(\mathbb{Z})$ for some $m\geq 1.$
\end{theorem}
\begin{proof} By Lemma~\ref{lem:5.4}, $\mathrm{JL}_n(R)$ is a nilpotent normal subgroup of $\mathrm{GL}_n(R),$ and $$\mathrm{GL}_n(R)/\mathrm{JL}_n(R)= \prod_{i=1}^{r}\mathrm{GL}_{n_i}(D_i)$$ for suitable positive integers $n_1, n_2,\ldots, n_r$ and division rings $D_1, D_2,\ldots, D_r.$ Also, $$G/(G\cap \mathrm{JL}_n(R))\cong G\cdot\mathrm{JL}_n(R)/\mathrm{JL}_n(R)$$ is a subgroup of $\mathrm{GL}_n(R)/\mathrm{JL}_n(R).$ Pick $N=G\cap \mathrm{JL}_n(R)$ and $L=G\cdot\mathrm{JL}_n(R)/\mathrm{JL}_n(R).$ Then, $N$ is normal nilpotent in $G$ and $G/N\cong L.$ Consider the projections $$\pi_i: \mathrm{GL}_n(R)/\mathrm{JL}_n(R)\longrightarrow \mathrm{GL}_{n_i}(D_i).$$ 	
	
Since $G$ is generalized radical ascendant in $\mathrm{GL}_n(R),$ the images $\pi_i(L)$ is generalized radical ascendant in $\mathrm{GL}_{n_i}(D_i).$ Then, according to~\cite[Theorem 3.8]{Pa_ChuaHai_2026}, $\pi_i(L)$ is central with $D_i$ is not a locally finite field in case $n_i\geq 2.$ If $D_i$ is a locally finite field, then clearly that $\pi_i(L)$ is a locally finite group. Hence, we conclude that $\pi_i(L)$ is abelian-by-(locally finite). In a similar way as in the proof of (2) in Theorem~\ref{th:5.7}, we get that $G$ is (nilpotent-by-abelian)-by-(locally finite), and $G$ is (nilpotent-by-polycyclic)-by-finite if $G$ is finitely generated, and that the quotient group $G/N$ can be embedded into the general linear group $\mathrm{GL}_n(\mathbb{Z})$ for some $m\geq 1.$
\end{proof}
\section{Locally generalized radical linear groups\\ over locally finite algebras}\label{S6}

Let $A$ be an algebra over a field $F$. Recall that an element $a\in A$ is \textit{algebraic} over $F$ if $a$ is a root of some non-zero polynomial $p[t]\in F[t]$. A non-empty subset $S$ of $A$ is \textit{algebraic} over $F$ if every its element is algebraic over $F$. If $A$ is algebraic over $F$, then sometimes we say that $A$ is an \textit{algebraic algebra}. Recall also that $A$ is \textit{locally finite} over $F$ if every its finite subset generates a finite-dimensional subalgebra over $F.$ Clearly, every locally finite algebra is algebraic. The converse is not true in general. However, I. Kaplansky and A. Shirshov showed that every algebraic $\mathrm{PI}$-algebra is locally finite (see, e.g. ~\cite[Ch.~10, \S12, Theorem 1]{Bo_Jacobson_1964}). In this section, we describe the structure of locally generalized radical linear groups over locally finite algebras. The first main result is Theorem~\ref{th:6.8} which shows that every locally generalized radical linear group over a locally finite algebra $A$ is locally (solvable-by-finite). Using this theorem we can give the affirmative answer to GBP for linear groups over locally finite algebras (see Theorem \ref{th:6.11}). As a consequence, the GBP holds for skew linear groups over weakly locally finite division rings.

The proof of the following lemma is similar to that of \cite[Lemma 4.1]{Pa_ChuaHai_2026}, so it can be omitted. 
\begin{lemma}\label{lem:6.1} Let $A$ be a finite-dimensional algebra over a field $F,$ and $n$ a positive integer. Then, every subgroup of the general linear group $\mathrm{GL}_n(A)$ can be embedded into the general linear group $\mathrm{GL}_{nr}(F)$, where $r=\dim_FA.$
\end{lemma}
In \cite{Zassenhaus_1938}, Zassenhaus proved that every locally solvable linear group over fields is solvable. Lemma~\ref{lem:6.1} allows us to get the same result for linear groups over finite-dimensional algebras. Namely, from Lemma~\ref{lem:6.1} together with ~\cite[Corollary 1.4.9]{Bo_Dixon_2017}, we can get the following theorem. 

\begin{theorem}\label{th:6.2} Let $A$ be a finite-dimensional $F$-algebra. Assume that $G$ a subgroup of the general linear group $\mathrm{GL}_n(A).$ The following conditions are equivalent:
	\begin{itemize}
		\item [(1)] $G$ is locally solvable.
		\item [(2)] $G$ is solvable.
		\item [(3)] $G$ is hyperabelian.
		\item [(4)] $G$ is radical.
		\item [(5)] $G$ is locally radical.
	\end{itemize} 
\end{theorem}
Combining Lemma~\ref{lem:6.1} and \cite[Theorems 1, 2]{Pa_Tits_1972} (see also e.g,~\cite[Theorem 1.4.3]{Bo_Dixon_2017}), we get the following result, which can be considered as Tits' Alternative for linear groups over finite-dimensional algebras.
\begin{theorem}\label{th:6.3} Let $A$ be a finite-dimensional $F$-algebra, $n$ a positive integer, and $G$ a subgroup of the general linear group $\mathrm{GL}_n(A).$ If $G$ contains no non-cyclic free subgroup, then $G$ contains a solvable normal subgroup $N$ such that $G/N$ is locally finite. Moreover, if $F$ has characteristic zero, then $G/N$ is finite.
\end{theorem}
\begin{corollary}\label{cor:6.4} Let $A$ be a finite-dimensional $F$-algebra, $n$ a positive integer, and $G$ a finitely generated subgroup of the general linear group $\mathrm{GL}_n(A).$ If $G$ contains no non-cyclic free subgroups, then $G$ is solvable-by-finite.
\end{corollary}
\begin{theorem}\label{th:6.5} Let $A$ be a finite-dimensional $F$-algebra, $n$ a positive integer, and $G$ a subgroup of the general linear group $\mathrm{GL}_n(A).$ If $G$ is locally generalized radical, then $G$ contains a solvable normal subgroup $N$ such that the quotient group $G/N$ is locally finite. Additionally, if $F$ has characteristic zero, then $G$ is solvable-by-finite.
\end{theorem}
\begin{proof}
Since $G$ is a locally generalized radical group, by \cite[Theorem 2.4]{Pa_ChuaHai_2026}, $G$ contains no non-cyclic free subgroups. Hence, the conclusions follows immediately from Theorem~\ref{th:6.3}.
\end{proof}
\begin{corollary}\label{cor:6.6} Let $A$ be a finite-dimensional $F$-algebra, $n$ a positive integer, and $G$ a subgroup of the general linear group $\mathrm{GL}_n(A).$ Then, the following conditions are equivalent:
	\begin{itemize}
		\rm\item [(1)]\textit{$G$ is solvable-by-(locally finite).}
		\item [(2)]\textit{$G$ is (locally solvable)-by-(locally finite).}
		\item [(3)]\textit{$G$ is locally (solvable-by-finite).}
		\item [(4)]\textit{$G$ is locally (hyperabelian-by-finite).}
		\item [(5)]\textit{$G$ is locally (radical-by-finite).}
		\item [(6)]\textit{$G$ is locally nearly radical.}
		\item [(7)]\textit{$G$ is generalized radical.}
		\item [(8)]\textit{$G$ is locally generalized radical.}
	\end{itemize}
\end{corollary}
\begin{proof}
	The implications (1) $\Rightarrow$ (2) $\Rightarrow$ (3) $\Rightarrow$ (4) $\Rightarrow$ (5) $\Rightarrow$ (6) $\Rightarrow$ (7) $\Rightarrow$ (8) are trivial, and the implication (8) $\Rightarrow$ (1) follows from Theorem \ref{th:6.5}.
\end{proof}
\begin{lemma}\label{lem:6.7} Let $A$ be a locally finite algebra over a field $F$, and $n$ a positive integer. If $H$ a finitely generated subgroup of the general linear group $\mathrm{GL}_n(A),$ then there exists a finite dimensional $F$-subalgebra $B$ of $A$ such that $H$ is contained in the general linear group $\mathrm{GL}_n(B).$ Moreover, either $H$ contains a non-cyclic free subgroup or $H$ is solvable-by-finite.
\end{lemma}
\begin{proof} Assume that $H=\langle x_1,x_2,\ldots, x_r\rangle$ is a finitely generated subgroup of the general linear group $\mathrm{GL}_n(A).$ Let $x_i=(a_{st}^i)$ and let $S$ be the set of all entries $a_{st}^i$ of all matrices $x_i$. Let $B$ be the subalgebra of  $A$ generated by the set $S$ over $F.$ By hypothesis, $B$ is a finite-dimensional algebra over $F$. Clearly, $H$ is a finitely generated subgroup of $\mathrm{GL}_n(B)$, and by Corollary~\ref{cor:6.4}, the result follows.
\end{proof}
\begin{theorem}\label{th:6.8} Let $A$ be a locally finite algebra over a field $F,$ $n$ a positive integer, and $G$ a subgroup of the general linear group $\mathrm{GL}_n(A).$ If $G$ contains no non-cyclic free subgroups, then $G$ is locally (solvable-by-finite). Consequently, every locally generalized radical linear group over a locally finite algebra is locally (solvable-by-finite).
\end{theorem}
\begin{proof} Assume that $H$ is a finitely generated subgroup of the group $G.$ Then, according to Lemma~\ref{lem:6.7}, $H$ is contained in the general linear group $\mathrm{GL}_{n}(B),$ where $B$ is a finite-dimensional subalgebra of $A.$ Moreover, either $H$ is solvable-by-finite or $H$ contains a non-cyclic free subgroup. The latter case cannot occur because by hypothesis, $G$ contains no non-cyclic free subgroups. Hence, $H$ is solvable-by-finite, and so $G$ is locally (solvable-by-finite). Consequently, by \cite[Theorem 2.4]{Pa_ChuaHai_2026}, the last statement of the theorem follows.
\end{proof}
\begin{corollary}\label{cor:6.9} Let $A$ be a locally finite algebra over a field $F,$ $n$ a positive integer, and $G$ a subgroup of the general linear group $\mathrm{GL}_n(A).$ Then, the following conditions are equivalent:
	\begin{itemize}
		\item [(1)] $G$ is locally (solvable-by-finite).
		\item [(2)] $G$ is locally (hyperabelian-by-finite).
		\item [(3)] $G$ is locally (radical-by-finite).
		\item [(4)] $G$ is locally nearly radical.
		\item [(5)] $G$ is locally generalized radical.
	\end{itemize}
\end{corollary}
Observe that every (locally solvable)-by-(locally finite) group is locally (solvable-by-finite). It is known that every  locally (solvable-by-finite) skew linear group over a locally finite division ring is (locally solvable)-by-(locally finite) (see \cite[3.3.9]{Pa_ShWe_86}). Hence, the following result follows immediately from Corollary \ref{cor:6.9}.
\begin{corollary}\label{cor:6.10}
Let $R$ be a locally finite division with center $F,$ $n$ a positive integer, and $G$ a subgroup of the general linear group $\mathrm{GL}_n(R).$ Then, the following conditions are equivalent:
	\begin{itemize}
		\item [(1)] $G$ is (locally solvable)-by-(locally finite).
		\item [(2)] $G$ is locally (solvable-by-finite).
		\item [(3)] $G$ is locally  (hyperabelian-by-finite).
		\item [(4)] $G$ is locally (radical-by-finite).
		\item [(5)] $G$ is locally nearly radical.
		\item [(6)] $G$ is locally generalized radical.
	\end{itemize}
\end{corollary}
 
\begin{question} Let $A$ be a locally finite algebra over a field $F$, $n$ a positive integer, and $G$ a subgroup of the general linear group $\mathrm{GL}_n(A)$. If $G$ is  a locally (solvable-by-finite), then is it true that $G$ is (locally solvable)-by-(locally finite)?
\end{question}
Now, we are ready to give the positive answer to 
the GBP for linear groups over locally finite algebras.
\begin{theorem}\label{th:6.11} Let $A$ be a locally finite algebra over a field $F$, $n$ a positive integer, and $G$ a subgroup of the general linear group $\mathrm{GL}_n(A)$. If $G$ is periodic, then $G$ is locally finite. 
\end{theorem}
\begin{proof} Assume that $G$ is a periodic subgroup of the general linear group $\mathrm{GL}_n(A).$ Since free groups can not be periodic, it follows from Theorem~\ref{th:6.8} that $G$ is locally (solvable-by-finite). By Corollary~\ref{cor:6.9} and Proposition~\ref{pro:4.8}, $G$ is locally finite. 
\end{proof}
\begin{corollary}\label{cor:6.12} Every periodic skew linear group over a weakly locally finite division ring is a locally finite group.
\end{corollary}
\begin{proof} Let $D$ be a weakly locally finite division ring, $n$ a positive integer, and $G$ a subgroup of the group $\mathrm{GL}_n(D)$. Suppose that $G$ is periodic. We have to show that $G$ is a locally finite group. Let $H=\langle x_1, x_2, \ldots, x_r\rangle$ be a finitely generated subgroup of $G$. Denote by $S$ the set of all entries of all matrices $x_1, x_2, \ldots, x_r$, and by $D_1$ the division subring of $D$ generated by $S$. Then, $D_1$ is a centrally finite division ring. Clearly, $H$ is a subgroup of the general skew linear group $\mathrm{GL}_n(D_1)$. Viewing $D_1$ as an algebra over its center $Z(D_1)$, by Theorem \ref{th:6.11}, $H$ is locally finite, and so it is finite. Hence, we conclude that $G$ is a locally finite group.
\end{proof}
\begin{theorem}\label{th:6.13} Let $A$ be a locally finite algebra over a field $F$, $n$ a positive integer, and $G$ a subgroup of the general linear group $\mathrm{GL}_n(A)$. If $G$ is finitely generated generalized radical, then $G$ has a subnormal series $$1\trianglelefteq N\trianglelefteq K\trianglelefteq H\trianglelefteq G$$ such that $N$ is nilpotent, $K/N$ is polycyclic, $H/K$ is finite, $G/H$ is finite, and $H$ is solvable. Additionally, if $G$ is irreducible, then $G$ is polycyclic-by-finite. Consequently, every finitely generated irreducible generalized radical linear group over a locally finite algebra can be embedded into the group $\mathrm{GL}_n(\mathbb{Z})$ for some $m\geq 1$, where $\mathbb{Z}$ is the ring of integers.
\end{theorem}
\begin{proof} Assume that $G$ is a finitely generated generalized radical subgroup of $\mathrm{GL}_n(A).$ According to Corollary~\ref{cor:6.9}, $G$ is solvable-by-finite.
Let $H$ be a solvable normal subgroup of $G$ such that $G/H$ is finite. Since $G$ is finitely generated and $H$ is a subgroup of finite index in $G$, it follows that $H$ is also finitely generated. In view of Lemmas~\ref{lem:6.7} and \ref{lem:6.1}, $H$ can be considered as a subgroup of the  general linear group $\mathrm{GL}_m(F)$ for some $m\geq 1.$ By~\cite[Theorem 1.4.7]{Bo_Dixon_2017}, $H$ is nilpotent-by-(abelian-by-finite). Then, $H$ contains a nilpotent normal subgroup $N$ such that $H/N$ is abelian-by-finite. Observe that every finitely generated abelian group is polycyclic, and $H/N$ is finitely generated. Hence, from the fact that $H/N$ is abelian-by-finite, it follows that $H/N$ is polycyclic-by-finite. Then, there exists a normal subgroup $K$ of $H$, containing $N$ such that $K/N$ is a polycyclic normal subgroup of $H/N$ and $$(H/N)/(K/N)\cong H/K$$
is finite. Hence, we get in $G$ the following series of subgroups
$$1\trianglelefteq N\trianglelefteq K\trianglelefteq H\trianglelefteq G,$$ 
where $N$ is nilpotent, $K/N$ is polycyclic, $H/K$ is finite, $G/H$ is finite, and $H$ is solvable.
Being a subgroup of a irreducible group, $H$ is irreducible. Then, according to~\cite[Theorem 1.4.6]{Bo_Dixon_2017}, $H$ contains an abelian normal subgroup $P$ such that the quotient group $H/P$ is finite. Since $H$ is finitely generated and $H/P$ is finite, it follows that $P$ is a finitely generated abelian group, and so $P$ is polycyclic. Hence, $H$ is polycyclic-by-finite. Recall that $G/H$ is finite, and observe that every finite group is obviously abelian-by-finite, and so is polycyclic-by-finite. It is known that the property of polycyclic-by-finite is preserved under extensions (\cite[Proposition~5.7]{Bo_SilberBAdd_2021}). Hence, it follows that $G$ is polycyclic-by-finite. The last statement follows from 
the Auslander-Swan Theorem (see e.g.,~\cite[3.3.1]{Bo_Lennox_2004}).
\end{proof}

We devote the remaining part of this section to study the automorphism group of a finitely generated linear group over a locally finite algebra.  Let $G$ be a group with its center $Z(G).$ Let $\mathrm{Aut}(G)$ be the automorphism group of $G.$ Recall that the inner automorphism group $\mathrm{Inn}(G)$ is normal in $\mathrm{Aut}(G),$ and $\mathrm{Inn}(G)$ is isomorphic to the group $G/Z(G).$ The factor group $\mathrm{Aut}(G)/\mathrm{Inn}(G)$ is said to be the \textit{outer automorphism group} of $G$ and is denoted by $\mathrm{Out}(G).$ 

\begin{theorem}\label{th:6.14} Let $A$ be a locally finite algebra over a field $F$, and $n$ a positive integer. Assume that $G$ is a finitely generated irreducible generalized radical subgroup of the general linear group $\mathrm{GL}_n(A).$ Then, the following statements hold:
	\begin{itemize}
		\item [(1)] If the automorphism group $\mathrm{Aut}(G)$ is periodic, then $\mathrm{Aut}(G)$ is finite, and so the derived subgroup $G'$ of $G$ is finite.
		\item [(2)] If $\mathrm{Out}(G)$ is periodic, then $\mathrm{Out}(G)$ is finite.
	\end{itemize}
\end{theorem}
\begin{proof} According to Theorem~\ref{th:6.13}, $G$ is polycyclic-by-finite. Then, either $\mathrm{Aut}(G)$ is polycyclic-by-finite or $\mathrm{Aut}(G)$ contains a non-cyclic free subgroup (see \cite{Pa_Eick_2003}). 
	
	(1) Suppose that the automorphism group 
	$\mathrm{Aut}(G)$ is periodic. Then, $\mathrm{Aut}(G)$ cannot contain non-cyclic free subgroups, and so it must be polycyclic-by-finite. In view of 	\cite[Lemma 6]{Zassenhaus_1969}, we conclude that $\mathrm{Aut}(G)$ is finite. Hence, the central factor group $G/Z(G)\cong \mathrm{Inn}(G)$ is finite, and by a theorem of Schur, the derived subgroup $G'$ of $G$ is finite, as required.
	
	(2) Suppose that the group $\mathrm{Out}(G)$ is periodic. We claim that the automorphism group $\mathrm{Aut}(G)$ contains no non-cyclic free subgroups. Suppose by contrary that $\mathrm{Aut}(G)$ contains a non-cyclic free subgroup, say $L$. Consider the following isomorphism
	$$L/(\mathrm{Inn}(G)\cap L)\cong \mathrm{Inn}(G)L/\mathrm{Inn}(G)\leq \mathrm{Out}(G).$$
	Since $G$ is polycyclic-by-finite, the group $\mathrm{Inn}(G)\cong G/Z(G)$ is polycyclic-by-finite, too. Then, $\mathrm{Inn}(G)\cap L$ is a polycyclic-by-finite free group. If $\mathrm{Inn}(G)\cap L$ is non-trivial, then, it follows from~\cite[Theorem 2.4]{Pa_ChuaHai_2026} that $\mathrm{Inn}(G)\cap L$ is cyclic. According to~\cite[Ch.1, Proposition~ 3.12]{Bo_LyPa_2001}, $\mathrm{Inn}(G)\cap L$ is a subgroup of finite index in $L$, and so $L$ is a cyclic-by-finite free group. Now, again by \cite[Theorem 2.4]{Pa_ChuaHai_2026}, $L$ is cyclic, a contradiction. Hence, $\mathrm{Inn}(G)\cap L$ must be trivial, which implies that $\mathrm{Out}(G)$ contains a non-cyclic free subgroup $\mathrm{Inn}(G)L/\mathrm{Inn}(G).$ But this is impossible because $\mathrm{Out}(G)$ is a periodic group, and so the claim is shown. Thus, the group $\mathrm{Aut}(G)$ is polycyclic-by-finite, and so $\mathrm{Out}(G)$ is polycyclic-by-finite, too. Now, in view of \cite[Lemma 6]{Zassenhaus_1969}, it follows that the group $\mathrm{Out}(G)$ is finite.
	
	The proof of the theorem is complete.
\end{proof}

In 1940, Mal'cev~\cite{Pa_mal'cev_1940} proved that every finitely generated linear group over a field is residually finite. The next theorem shows that the same result also holds for a finitely generated linear group over a locally finite algebra. Moreover, its automorphism group is residually finite.
\begin{theorem}\label{th:6.15} Let $A$ be a locally finite algebra over a field $F$, and $n$ a positive integer. Assume that $G$ is a finitely generated subgroup of the general linear group $\mathrm{GL}_n(A).$ Then, the following statements hold:
\end{theorem}
\begin{itemize}
\item [(1)] $G$ is residually finite.
\item [(2)] $\mathrm{Inn}(G)$ is residually finite.
\item [(3)] $\mathrm{Aut}(G)$ is residually finite.
\item [(4)] $G$ and $\mathrm{Inn}(G)$ are Hopfian.
\end{itemize}
\begin{proof} Assume that $G$ is a finitely generated subgroup of $\mathrm{GL}_n(A).$ Then, according to Lemma~\ref{lem:6.7} and Lemma~\ref{lem:6.1}, $G$ can be embedded into the general linear group $\mathrm{GL}_m(F)$ for some $m\geq 1.$ Thus, by a theorem of Mal'cev~\cite{Pa_mal'cev_1940}, $G$ is residually finite. By a theorem of G. Baumslag~\cite[Theorem 1]{Pa_Baumslag_1963}, $\mathrm{Aut}(G)$ is residually finite. Being a subgroup of a residually finite group, the group $\mathrm{Inn}(G)$ is residually finite, too. The final assertion of the theorem immediately follows from~\cite[Theorem 2.4.3]{Bo_SilberCoor_2010}.
\end{proof}
\section{Locally generalized radical linear groups over $\mathrm{PI}$-algebras}\label{S7}
Let $R$ be a ring. Recall that the \textit{Levitzky radical} of $R$, denoted by $L(R)$, is the sum of all locally nilpotent ideals of $R$. It is known that this sum is also a locally nilpotent ideal of $R$ (see, e.g. \cite[Proposition 1, p.197]{Bo_Jacobson_1964}). Hence, $L(R)$ is a unique maximal locally nilpotent ideal of $R$. Observe that the Levitzky radical $L(R)$ is contained in the Jacobson radical $J(R)$. Indeed, if $a\in L(R),$ then $a$ is a nilpotent element, which implies that $1+a$ is invertible, and so $a\in J(R)$. 

Now, let $A$ be an algebra over a field $F,$ and let $F[x_1,x_2,\ldots, x_r]$ be the free algebra generated by $r$ indeterminates over $F.$ The algebra $A$ is said to be a \textit{$\mathrm{PI}$-algebra} if there exists a non-zero element $f\in F[x_1,x_2,\ldots, x_r]$ such that $f(a_1,a_2,\ldots, a_r)=0$ for all $a_i\in A.$ It is easy to see that if $A$ is a $\mathrm{PI}$-algebra, then the matrix ring $\mathrm{M}_n(A)$ is also a $\mathrm{PI}$-algebra for each $n\in\mathbb{N}.$ 
\begin{lemma}\label{lem:7.1} Let $A$ be a $\mathrm{PI}$-algebra, and $L=L(A)$ its Levitzky radical.  Then, $$G(L):=1+L=\{1+a\mid a\in L\}$$ is a locally nilpotent normal subgroup of the unit group $A^\times$ and $$A^\times/G(L)\cong (A/L)^\times.$$ Moreover, the following assertions hold:
\begin{itemize}
\item [(1)] The quotient group $A^\times/G(L)$ can be embedded into the general linear group $\mathrm{GL}_m(R)$ of degree $m\geq 1$ over a semiprime commutative ring $R.$
\item [(2)] If $H$ is a subgroup of $A^\times,$ then $H$ contains a locally nilpotent normal subgroup $N$ such that every finitely generated subgroup of $H/N$ can be embedded into a direct product $\prod_{i=1}^{r}\mathrm{GL}_m(F_i),$ where $F_i$ is a field for all $i$.
     \end{itemize}
\end{lemma} 
\begin{proof} Let $L$ be the Levitzky radical of a PI-algebra $A$. Since $L$ is locally nilpotent, it follows that $L$ is a nil ideal of $A$. Then, in view of Lemma~\ref{lem:5.1}, $G(L)$ is a normal subgroup of the unit group $A^\times$, and $A^\times/G(L)\cong (A/L)^\times.$

Assume that $$G_0=\langle 1+a_1,1+a_2,\ldots, 1+a_s\rangle$$
is a finitely generated subgroup of the group $G(L),$ where $a_1,a_2,\ldots, a_s\in L.$ Then, $I=\langle a_1,a_2,\ldots, a_s\rangle$ is a nilpotent ideal of $A.$ According to Lemma~\ref{lem:5.1}, $G(I)$ is a nilpotent normal subgroup of $A^\times$.  Being a subgroup of a nilpotent group $G(I)$, $G_0$ is also nilpotent, and so $G(L)$ is a locally nilpotent group.
	
(1) It is known that the Levitzky radical of $A$ is a semiprime ideal of $A$. Hence, the quotient algebra $A/L$ is semiprime, and so $A/L$ has no non-zero nilpotent ideal (see, e.g. ~\cite[(10.16) Proposition]{Bo_Lam_2001}). Being a factor algebra of a $\mathrm{PI}$-algebra, $A/L$ is a $\mathrm{PI}$-algebra. Then, in view of~\cite[\S 6, Theorem 2]{Bo_Jacobson_1964}, $A/L$ can be embedded in the matrix ring $\mathrm{M}_m(R)$ of degree $m$ over a semiprime commutative ring $R$ (see also the proof of \cite[Theorem~2]{Pa_Procesi_1966}). Since $A^\times/G(L)\cong (A/L)^\times,$ it follows that the quotient group $A^\times/G(L)$ can be embedded in the general linear group $\mathrm{GL}_m(R).$
	
(2) Assume that $H$ is a subgroup of $A^\times$, and put $N=H\cap G(L)$. Then, $N$ is a locally nilpotent normal subgroup of $H,$ and  
$$\overline{H}:=H/N\cong G(L)H/G(L).$$ 
In view of Part (1), $\overline{H}$ can be considered as a subgroup of the group $\mathrm{GL}_m(R),$ where $R$ is a semiprime commutative ring. Now, let $\overline{H}_0=\langle \overline{h}_1,\overline{h}_2,\ldots, \overline{h}_t\rangle$ be a finitely generated subgroup of $\overline{H}$, and $K$  be the subring of $R$ generated by all entries of all matrices $\overline{h}_1,\overline{h}_2,\ldots, \overline{h}_t.$ Then, $K$ is a Noetherian ring containing no  non-zero nilpotent ideals, and so $K$ is a semiprime commutative ring. Let $P_1, P_2,\ldots, P_r$ be all prime ideals of $K$, and $F_i$ is the quotient field of $K/P_i$ for all $i$.  Then, $0=\bigcap_{i=1}^r P_i$ and $\mathrm{GL}_m(K)$ is embedded in the finite direct product $\prod_{i=1}^{r}\mathrm{GL}_m(F_i).$ Observe that $\overline{H}_0\le \mathrm{GL}_m(K).$ Consequently, $\overline{H}_0$ can be embedded into $\prod_{i=1}^{r}\mathrm{GL}_m(F_i)$, and  the proof of the lemma is now complete. 
\end{proof}
The following theorem is the main result of this section.
\begin{theorem}\label{th:7.2} Let $A$ be a $\mathrm{PI}$-algebra over a field $F,$ and $n$ a positive integer. Assume that $G$ is a locally generalized radical subgroup of the general linear group $\mathrm{GL}_n(A).$ Then, the following assertions hold:
	\begin{itemize}
		\item [(1)] $G$ contains a locally nilpotent normal subgroup $N$ such that the quotient group $G/N$ is locally (solvable-by-finite). Additionally, if the group $G$ is finitely generated, then $G$ has a finite normal series 
			$$1\trianglelefteq N\trianglelefteq H\trianglelefteq G$$
		 such that $N$ is locally nilpotent, $H/N$ is solvable, and $G/H$ is finite.
		\item [(2)] Assume additionally that $A$  is affine, then $G$ is locally (solvable-by-finite).
		\item [(3)] Assume additionally that $A$ is algebraic affine, then $G$ is solvable-by-(locally finite).
	\end{itemize}	
\end{theorem}
\begin{proof} Let $A$ be a PI-algebra over $F$. Observe that the matrix algebra $\mathrm{M}_n(A)$ is also a PI-algebra over $F$. Let $L=L(\mathrm{M}_n(A))$ be the Levitzky radical of $\mathrm{M}_n(A)$.  Assume that $G$ is a locally generalized radical of the general linear group $\mathrm{GL}_n(A).$ 
	
(1) Set $N=G\cap G(L)$. According to the proof of Part (2) of Lemma~\ref{lem:7.1}, $N$ is a locally nilpotent normal subgroup of $G$ and every finitely generated subgroup $\overline{G}$ of the quotient group $G/N$ can be embedded into a finite direct product $\prod_{i=1}^{r}\mathrm{GL}_m(F_i)$ with $F_i$ is a field for all $i$. For every $i\in\{1,2,\ldots, r\},$ consider the canonical projections 
	$$\pi_i: \prod_{i=1}^{r}\mathrm{GL}_m(F_i)\longrightarrow \mathrm{GL}_m(F_i).$$ Since $G$ is locally generalized radical, it follows that the quotient group $G/N$ is also locally generalized radical, and so $\overline{G}$ is finitely generated generalized radical. Thus, $\pi_i(\overline{G})$ is finitely generated generalized radical in the general linear group $\mathrm{GL}_m(F_i).$ In view of \cite[Corollary 1.4.5]{Bo_Dixon_2017}, it follows that $\pi_i(\overline{G})$ is solvable-by-finite, and so the direct product $P=\prod_{i=1}^{r}\pi_i(\overline{G})$ is solvable-by-finite. Since  $\overline{G}$ is a subgroup of $P,$ it follows that $\overline{G}$ is solvable-by-finite, too. Consequently, we conclude that the quotient group $G/N$ is locally (solvable-by-finite). Now, if $G$ is finitely generated, then $G/N$ is also finitely generated, and so it is solvable-by-finite. Hence, there exists a normal subgroup $H$ of $G$ containing $N$ such that $H/N$ is solvable and $G/H$ is finite.
	Thus, $G$ has a finite normal series $$1\trianglelefteq N\trianglelefteq H\trianglelefteq G$$ such that $N$ is locally nilpotent, $H/N$ is solvable, and $G/H$ is finite. 
	
(2) By (1), $N=G\cap G(L)$ is a locally nilpotent normal subgroup of $G$ and $G/N$ is locally (solvable-by-finite). Assume additionally that $A$ is affine. Then, by the Braun-Kemer-Razmyslov Theorem (see~\cite[Theorem~ 1.2]{Pa_BelovRowen_2014}), the Jacobson radical $J(A)$ is nilpotent, and so by Lemma~\ref{lem:5.2}, the Jacobson group $\mathrm{JL}_n(A)$ is nilpotent. Since $L=L(\mathrm{M}_n(A))\subseteq J(\mathrm{M}_n(A)),$ it follows that $G(L)\subseteq \mathrm{JL}_n(A),$ and so $G(L)$ is nilpotent, which implies that $N$ is nilpotent. Finally, because $G/N$ is locally (solvable-by-finite), it can be easily to see that $G$ is locally (solvable-by-finite).
	
(3) The result follows  immediately from~\cite[Corollary 3]{Pa_BelovRowenVishne_2006} and Theorem~\ref{th:6.5}.
\end{proof}
\begin{corollary}\label{cor:7.3} Let $A$ be an affine $\mathrm{PI}$-algebra over a field $F,$  $n$ a positive integer, and $G$ a subgroup of the general linear group $\mathrm{GL}_n(A).$ Then, the following conditions are equivalent:
	\begin{itemize}
		\item [(1)] $G$ is locally (solvable-by-finite).
		\item [(2)] $G$ is locally (hyperabelian-by-finite).
		\item [(3)] $G$ is locally (radical-by-finite).
		\item [(4)] $G$ is locally nearly radical.
		\item [(5)] $G$ is locally generalized radical.
	\end{itemize}
\end{corollary}
\begin{proof} The result  follows immediately from the conclusion (2) of Theorem~\ref{th:7.2}.
\end{proof}
\begin{theorem}\label{th:7.4} Let $A$ be an algebraic $\mathrm{PI}$-algebra over a field $F,$  $n$ a positive integer, and $G$ a subgroup of the general linear group $\mathrm{GL}_n(A)$. If $G$ is a locally generalized radical, then $G$ is locally (solvable-by-finite). Suppose further that $G$ is finitely generated.  Then, the following assertions hold:
	\begin{itemize}
		\item [(1)] $G$ has a subnormal series $$1\trianglelefteq N\trianglelefteq K\trianglelefteq H\trianglelefteq G$$
		such that $N$ is nilpotent, $K/N$ is polycyclic, $H/K$ is finite, $G/H$ is finite, and $H$ is solvable.
		\item [(2)] Moreover, if $G$ is additionally irreducible, then 
		  \begin{itemize}
		  	\item [(a)] $G$ is polycyclic-by-finite.
		  	\item [(b)] $\mathrm{Aut}(G)$ is periodic if and only if $\mathrm{Aut}(G)$ is finite. In this case, the derived subgroup $G'$ of $G$ is finite.
		  	\item [(c)] $\mathrm{Out}(G)$ is periodic if and only if $\mathrm{Out}(G)$ is finite.
		  \end{itemize}
	\end{itemize}
\end{theorem}
\begin{proof} Assume that $A$ is an algebraic $\mathrm{PI}$-algebra over a field $F.$ Then, by a result of I. Kaplansky and A. Shirshov, $A$ is a locally finite algebra (see, e.g. ~\cite[Ch.~10, \S12, Theorem 1]{Bo_Jacobson_1964}). If $G$ is a locally generalized radical subgroup of the  group $\mathrm{GL}_n(A),$ then $G$ is locally (solvable-by-finite) by Theorem~\ref{th:6.8}. The remaining assertions of the theorem follow from Theorems~\ref{th:6.13} and~\ref{th:6.14}.
\end{proof}
\begin{theorem}\label{th:7.5} Let $A$ be a $\mathrm{PI}$-algebra over a field $F$, $n$ a positive integer, and $G$ a subgroup of the general linear group $\mathrm{GL}_n(A).$ If $G$ is finitely generated, then the following statements hold:
	\begin{itemize}
		\item [(1)] $G$ is (locally nilpotent)-by-(residually finite).
		\item [(2)] If $A$ is affine, then $G$ is nilpotent-by-(residually finite). 
		\item [(3)] If $A$ is algebraic, then $G$ is residually finite, and so both the groups $\mathrm{Inn}(G)$ and $\mathrm{Aut}(G)$ are residually finite. Moreover, $G$ and $\mathrm{Inn}(G)$ are Hopfian. 
	\end{itemize}	
\end{theorem}
\begin{proof}  Assume that $G$ is a finitely generated subgroup of the general linear group $\mathrm{GL}_n(A).$ Let $L=L(\mathrm{M}_n(A))$ be the Levitzky radical of the matrix algebra $\mathrm{M}_n(A)$.
	
(1) In view of Part (2) of Lemma~\ref{lem:7.1}, it follows that $G$ contains a locally nilpotent normal subgroup $N=G\cap L$ such that $G/N$ can be embedded into a finite direct product $\prod_{i=1}^{r}\mathrm{GL}_m(F_i)$, where  $F_i$ is a field for all $i$. For every $i\in\{1,2,\ldots, r\},$ let $$\pi_i: \prod_{i=1}^{r}\mathrm{GL}_m(F_i)\longrightarrow \mathrm{GL}_m(F_i)$$ be the canonical projection from the product $\prod_{i=1}^{r}\mathrm{GL}_m(F_i)$ to the group $\mathrm{GL}_m(F_i)$. Observe that $\pi_i(G/N)$ is finitely generated in the general linear group $\mathrm{GL}_m(F_i).$ Hence, by a theorem of Malcev~\cite{Pa_mal'cev_1940}, $\pi_i(G/N)$ is residually finite, and according to~\cite[Proposition 2.2.2]{Bo_SilberCoor_2010}, the direct product $P=\prod_{i=1}^{r} \pi_i(G/N)$ is residually finite. Note that $G/N$ is a subgroup of $P.$ Being a subgroup of a residually finite group, $G/N$ is residually finite. Hence, $G$ is (locally nilpotent)-by-(residually finite). 
	
(2) Assume that $A$ is an affine $\mathrm{PI}$-algebra over $F.$ Then, by the Braun-Kemer-Razmyslov Theorem (see~\cite[Theorem~ 1.2]{Pa_BelovRowen_2014}), the Jacobson radical $J(A)$ of $A$ is nilpotent, and in view of Lemma~\ref{lem:5.3}, $\mathrm{JL}_n(A)$ is nilpotent. Since the Levitzki radical $L(A)\subseteq J(A),$ it follows that  
$$L=L(\mathrm{M}_n(A))\subseteq J(\mathrm{M}_n(A)).$$
This shows that $G(L)\subseteq \mathrm{JL}_n(A),$ and so $G(L)$ is nilpotent, which implies that $N$ is nilpotent. Consequently, the group $G$ is nilpotent-by-(residually finite).
	
(3) Observe that if $A$ is algebraic, then, in view of~\cite[Ch.~10, \S12, Theorem 1]{Bo_Jacobson_1964}, $A$ is a locally finite algebra, and the result follows from Theorem~\ref{th:6.15}. 

The proof of the theorem is now complete. 
\end{proof}
\section{Noetherian linear groups over $\mathrm{PI}$-algebras}\label{S8}

In this section, we show that Baer's Conjecture holds for locally generalized radical groups. More precisely, we prove that every generalized radical noetherian group is polycyclic-by-finite. Also, we show that  Baer's Conjecture holds for linear groups over locally finite algebras and PI-algebras. As a consequence, we can re-obtain a theorem of Zassenhaus~\cite[Theorem 1]{Zassenhaus_1969}.

The following lemma is obvious.
\begin{lemma}\label{lem:8.1} Let $G$ be a group, $H\leq G$ and $N\trianglelefteq G.$ The following conditions hold:
	\begin{itemize}
		\rm\item [(1)]\textit{If $G$ is noetherian, then so is $H.$}
		\item [(2)]\textit{$G$ is noetherian if and only if $N$ and $G/N$ are noetherian.}
	\end{itemize}
\end{lemma}
\begin{lemma}\label{lem:8.2} The following conditions are equivalent for a noetherian group $G.$
	\begin{itemize} 
		\rm\item [(1)]\textit{$G$ is polycyclic.}
		\item [(2)]\textit{$G$ is solvable}
		\item [(3)]\textit{$G$ is hyperabelian.}
		\item [(4)]\textit{$G$ is radical.}  
	\end{itemize}
\end{lemma}
\begin{proof} The implications $(1)\Rightarrow (2) \Rightarrow (3)\Rightarrow (4)$ are trivial, while the implication $(4)\Rightarrow (1)$ follows from~\cite[Lemma 3.1]{Pa_ChNaHa_25}.
\end{proof}
\begin{theorem}\label{th:8.3} Every generalized radical noetherian group is polycyclic-by-finite. Consequently, every generalized radical noetherian group can be embeded into the integer general linear group $\mathrm{GL}_n(\mathbb{Z})$ for some $n\geq 1.$
\end{theorem}
\begin{proof} Suppose that $G$ is a generalized radical noetherian group. Since every subgroup of $G$ is also noetherian, the group $G$ has an ascending series
	$$1=H_0\le H_1\le\cdots H_\alpha\le H_{\alpha+1}\le\cdots H_\gamma=G$$
	whose factors are either nilpotent or finite.
	We proceed by transfinite induction on $\gamma.$ If $\gamma=1,$ then either $G$ is nilpotent or finite. In the first case, it follows from Lemma~\ref{lem:8.2} that $G$ is polycyclic. In the latter case, it follows that $G$ is finite, and so $G$ trivially is abelian-by-finite. According to Lemma~\ref{lem:8.2}, it follows that $G$ is a polycyclic-by-finite group. Hence, the conclusion follows for $\gamma=1$.
	
	Now, let $\gamma>1$, and consider an ordinal $\alpha\le \gamma$. Suppose inductively that $H_\beta$ is polycyclic-by-finite for all $\beta$ less than $\alpha$. We have to prove that $H_\alpha$ is also polycyclic-by-finite. If $\alpha$ is a limit ordinal, then 
	$$H_\alpha=\bigcup_{\beta<\alpha} H_\beta.$$
	Since $H_\alpha$ is finitely generated, there exists some ordinal $\beta_0<\alpha$ such that $H_\alpha=H_{\beta_0}$. By inductive hypothesis, we conclude that $H_\alpha$ is polycyclic-by-finite. If $\alpha$ is not a limit ordinal, then there exists $\alpha-1$, and $H_\alpha/H_{\alpha-1}$ is either nilpotent or finite. Hence, $H_\alpha/H_{\alpha-1}$ is either nilpotent or abelian-by-finite. In view of Lemmas~\ref{lem:8.1} and~\ref{lem:8.2}, it follows that the factor $H_\alpha/H_{\alpha-1}$ is polycyclic-by-finite. By inductive hypothesis, $H_{\alpha-1}$ is polycyclic-by-finite. Hence, in view of \cite[Proposition 5.7]{Bo_SilberBAdd_2021}, it follows that $H_\alpha$ is polycyclic-by-finite. Consequently, $G$ is polycyclic-by-finite. Thus, by the Auslander–Swan Theorem (see e.g.~\cite[Theorem 5.22]{Bo_SilberBAdd_2021}), $G$ embeds into the integer general linear group $\mathrm{GL}_n(\mathbb{Z})$ for some $n\geq 1,$ as required.
\end{proof}
\begin{corollary}\label{cor:8.4} The following conditions are equivalent for a noetherian group $G.$ 
	\begin{itemize}
		\rm\item [(1)]\textit{$G$ is polycyclic-by-finite.}
		\item [(2)]\textit{$G$ is solvable-by-finite}
		\item [(3)]\textit{$G$ is hyperabelian-by-finite.}
		\item [(4)]\textit{$G$ is radical-by-finite.}
		\item [(5)]\textit{$G$ is nearly radical.}
		\item [(6)]\textit{$G$ is generalized radical.}
	\end{itemize}
\end{corollary}
\begin{proof} The result follows from Lemma~\ref{lem:8.2} and Theorem~\ref{th:8.3}.
\end{proof}
\begin{theorem}\label{th:8.5} A noetherian group contains no non-cyclic free subgroups. In particular, every free noetherian group is a cyclic group. 
\end{theorem}
\begin{proof} Suppose by contrary that $G$ is a noetherian group containing a non-cyclic free subgroup $L$. Hence, according to~\cite[Proposition 3.1]{Bo_LyPa_2001}, $L$ contains a free subgroup $N$ of infinite free rank. But this is impossible because $N$ is a noetherian group. The last statement is obvious.
\end{proof}
\begin{corollary}\label{cor:8.6} If $G$ is a free group, then the following conditions are equivalent:
	\begin{enumerate}
		\rm\item\textit{$G$ is noetherian.}
		\rm\item\textit{$G$ is cyclic.}
		\rm\item\textit{$G$ is cyclic-by-finite.}
		\rm\item\textit{$G$ is polycyclic-by-finite.}
		\rm\item\textit{$G$ is (locally solvable)-by-finite.}
		\rm\item\textit{$G$ is (locally solvable)-by-(locally finite).}
		\rm\item\textit{$G$ is locally (solvable-by-finite).}
		\rm\item\textit{$G$ is hyperabelian-by-finite.}
		\rm\item\textit{$G$ is radical-by-finite.}
		\rm\item\textit{$G$ is nearly radical.}
		\rm\item\textit{$G$ is generalized radical.}
		\rm\item\textit{$G$ is locally generalized radical.}
	\end{enumerate}
\end{corollary}
\begin{proof} The result follows from Theorem~\ref{th:8.5} and~\cite[Theorem 2.4]{Pa_ChuaHai_2026}.
\end{proof}
The next theorem shows that Baer's Conjecture holds for linear groups over a locally finite algebra.
\begin{theorem}\label{th:8.7} Let $A$ be a locally finite algebra over a field $F.$ If $G$ is a noetherian subgroup of the general linear group $\mathrm{GL}_n(A),$ then $G$ is polycyclic-by-finite.
\end{theorem}
\begin{proof} Assume that $G$ is a noetherian subgroup of the general linear group $\mathrm{GL}_n(A).$ By Theorem~\ref{th:8.5}, a noetherian group cannot contain non-cyclic free subgroups. In view of Theorem \ref{th:6.8}, it follows that the group $G$ is solvable-by-finite. Hence, according to Corollary~\ref{cor:8.4}, $G$ is polycyclic-by-finite, as required.
\end{proof}
Since the matrix ring $\mathrm{M}_n(F)$ over a field $F$ is an algebra of dimension $n^2$ over $F.$ From Theorem~\ref{th:8.7}, we re-obtain a theorem of Zassenhaus~\cite[Theorem 1]{Zassenhaus_1969}.
\begin{theorem}\label{th:8.8}\textnormal{(Zassenhaus~\cite{Zassenhaus_1969})} Every noetherian linear group over a field is polycyclic-by-finite.
\end{theorem}
The next theorem gives the affirmative answer to Bear's Conjecture for noetherian linear groups over $\mathrm{PI}$-algebras.
\begin{theorem}\label{th:8.9} Let $A$ be a $\mathrm{PI}$-algebra, $n$ a positive integer, and $G$ a subgroup of the general linear group $\mathrm{GL}_n(A).$ If $G$ is noetherian, then $G$ is polycyclic-by-finite. 
\end{theorem}
\begin{proof} Observe that if $A$ is a PI-algebra then so is the matrix algebra $\mathrm{M}_n(A)$. Assume that $G$ is a noetherian subgroup of the general linear group $\mathrm{GL}_n(A).$ According to Lemma~\ref{lem:7.1}, $G$ contains a locally nilpotent normal subgroup $N$ such that the quotient group $G/N$ can be embedded into a finite direct product 
	$$\prod_{i=1}^{r}\mathrm{GL}_m(F_i),$$ 
	where $F_i$ is a field for all $i$. For every $i\in\{1,2,\ldots, r\},$ let 
	$$\pi_i: \prod_{i=1}^{r}\mathrm{GL}_m(F_i)\longrightarrow \mathrm{GL}_m(F_i)$$ be the canonical projection of the product $\prod_{i=1}^{r}\mathrm{GL}_m(F_i)$ to the group $\mathrm{GL}_m(F_i)$. Since $G$ is noetherian, the quotient group $G/N$ is noetherian (see Lemma~\ref{lem:8.1}), and it follows that $\pi_i(G/N)$ is noetherian in the general linear group $\mathrm{GL}_m(F_i).$ Thus, according to Theorem~\ref{th:8.8}, $\pi_i(G/N)$ is polycyclic-by-finite, and so the product $P=\prod_{i=1}^{r}\pi_i(G/N)$ is polycyclic-by-finite. Being a subgroup of a polycyclic-by-finite group,  $G/N$ is also polycyclic-by-finite. On the other hand, since $G$ is noetherian, so is the subgroup $N$ (see Lemma~\ref{lem:8.1}). Hence, $N$ is nilpotent noetherian group, and it follows from Lemma~\ref{lem:8.2} that $N$ is polycyclic. Since $N$ is polycyclic and $G/N$ is polycyclic-by-finite, $G$ is polycyclic-by-finite, as required.
\end{proof}

As an application of Theorem~\ref{th:8.9}, we get the following result.
\begin{corollary}\label{cor:8.10} Let $G$ be a noetherian group, and let $$\{\sigma_{i}: G\longrightarrow \mathrm{M}_{n_i}(F_i)\mid F_i \text{ is a field }, i\in I\}$$ be a family of matrix representations satisfying the following conditions:
	\begin{itemize} 
		\item [(1)] $\mathrm{Ker}(\sigma_i)=1$ for all $i\in I$.
		\item [(2)] The dimensions of the representations $\sigma_i$ are bounded.
	\end{itemize}
	Then, $G$ is polycyclic-by-finite.
\end{corollary}
\begin{proof} By (2), there exists a positive integer $n$ such that $n_i\leq n$ for all $i\in I.$ Now, by (1), $G$ can be embedded into $A=\prod_{i\in I}\mathrm{M}_{n_i}(F_i).$ Again by (2), we get that $A$ is a $\mathrm{PI}$-algebra, and so in view of Theorem~\ref{th:8.9}, $G$ is polycyclic-by-finite as required. 
\end{proof}
\bigskip

{\noindent\textbf{Funding} }This research is funded by Vietnam National University of Ho Chi Minh City (VNU-HCM) under grant number C2026-16-05

\bigskip 

{\noindent\textbf{Conflict of Interest.} }The authors have no conflict of interest to declare that are relevant to this article.

\bigskip

{\noindent\textbf{Data availability}}
\bigskip

No data was used for the research described in the article.


\begin{thebibliography}{}


\bibitem{Pa_Baer_1956} R. Baer, Noethersche Gruppen I. \textit{Math. Z}, \textbf{66} (1956), 269-288. II. \textit{Math. Ann}, \textbf{165} (1966), 163-180.


\bibitem{Pa_BelovRowen_2014} A. K. Belov and L. H. Rowen, The Braun-Kemer-Razmyslov theorem for affine $\mathrm{PI}$-algebras, arXiv:1405.0730vl[math.RA] 4 May 2014.

\bibitem{Pa_BelovRowenVishne_2006} A. K. Belov, L. H. Rowen, and U. Vishne, Normal bases of $\mathrm{PI}$-algebras. \textit{Adv. Appl. Math.}, \textbf{37} (2006), 378-389.


\bibitem{Pa_Baumslag_1963} G. Baumslag, Automorphism groups of residually finite groups, \textit{J. London Math. Soc.} \textbf{38} (1963), 117–118.

\bibitem{Pa_Burn_1902} W. Burnside, On an unsettled question in the theory of discontinuous groups, \textit{Q. J. Pure Appl. Math.} \textbf{33} (1902), 230--238.	
	
	
\bibitem{Pa_ChuaHai_2026} L. V. Chua, and B. X. Hai, Locally generalized radical skew linear groups, \textit{J. Algebra} \textbf{697} (2026) 854--873, https://doi.org/10.1016/j.jalgebra.2026.03.006.
	
\bibitem{Pa_ChNaHa_24} L. V. Chua, C. M. Nam, and B. X. Hai, Skew linear groups with rank conditions, \textit{J. Algebra}, \textbf{653} (2024), 200-219.

\bibitem{Pa_ChNaHa_25} L. V. Chua, C. M. Nam, and B. X. Hai, On the solvable radical of groups, \textit{J. Algebra}, \textbf{682} (2025), 787--803, https://doi.org/10.1016/j.jalgebra.2025.06.018

	
	
\bibitem{Pa_CoeMil_1998} S. P. Coelho and C. P. Milies, Torsion subgroups of units in Artinian rings, \textit{Lecture notes in pure and applied mathematics} \textbf{198} (1998), 75--81.
	
\bibitem{Bo_Cohn_1995} P. M. Cohn, \textit{Skew fields}, Theory of general division ring, Cambridge University, 1995.	

\bibitem{Pa_Danh-Deo_2023} L. Q. Danh, T. T. Deo, A note on subnormal subgroups in division rings containing solvable subgroups, \textit{Bull. Austr. Math. Soc.}, \textbf{108}:3 (2023), 422--427. https://doi.org/10.1017/S0004972722001599

\bibitem{Pa_Danh-Khanh_2021} L. Q. Danh, H. V. Khanh, Locally solvable subnormal and quasinormal subgroups in division rings, \textit{Hiroshima Math. J.}, \textbf{51} (2021), 267--274, doi:10.32917/h2020034
		
\bibitem{Pa_dbh_2019} T. T. Deo, M. H. Bien, B. X. Hai, On weakly locally finite division rings, \textit{Acta Mathematica Vietnamica} (2019) \textbf{44}, 553--569.
	
	
\bibitem{Bo_Dixon_2017} M. R. Dixon, L. A. Kurdachenko and I. Ya. Subbotin, \textit{Ranks of groups, The tools, characteristics, and restrictions}, Wiley, 2017.
	
\bibitem{Bo_Draxl_1983} P. K. Draxl, \textit{Skew fields}, Cambridge University Press, 1983.
	
\bibitem{Pa_Golod_1964} E. S. Golod, On nil-algebras and finitely approximable p-groups, \textit{Izv. Akad. Nauk SSSR, Ser. Mat.} \textbf{28} (1964), 273--276.
	
\bibitem{Pa_Golod-Shaf_1964} E. S. Golod and I. R. Shafarevich, On the class field tower, \textit{Izv. Akad. Nauk SSSR Ser. Mat.} \textbf{28} (1964), 261--272.
	
\bibitem{Pa_Eick_2003} B. Eick, When is the automorphism group of virtually polycyclic group virtually polycyclic?  \textit{Glasgow. Math. J.} \textbf{45} (2003), 527--533.

\bibitem{Bo_Jacobson_1964} N. Jacobson, \textit{Structure of Rings}, Amer. Math. Soc., Providence (1964).

\bibitem{Pa_Khanh-Hai_2022} H. V. Khanh, B. X. Hai, Locally solvable and solvable-by-finite maximal subgroups of $\mathrm{GL}_n(D)$, \textit{Publ. Mat.} \textbf{66} (2022), 77--97, DOI: 10.5565/PUBLMAT6612203

\bibitem{Pa_HaiChua_2025} B. X. Hai and L. V. Chua, Skew linear groups with rank conditions, II \textit{J. Algebra}, \textbf{677} (2025), 430-447., https://doi.org/10.1016/j.jalgebra.2025.03.054
	
\bibitem{Pa_Herstein_1978} I. N. Herstein, Multiplicative commutators in division rings \textit{Israel J. Math.} (2) \textbf{31} (1978), 180--188.
	
\bibitem{Pa_He_80} I. N. Herstein, Multiplicative commutators in division ring II, \textit{Rendiconti Del Circolo Matematico di Palermo Serie II, Tomo XXIX} (1980), 485-489.
	
\bibitem{Pa_HersteinScott_1963} I. N. Herstein and W. R. Scott, Subnormal subgroups of division rings \textit{Can. J. Math.} \textbf{15} (1963), 80--83.
	
\bibitem{Pa_Hua_1950} L. K. Hua, On the multiplicative group of a field, \textit{Acad. Sinic. Sci. Record} \textbf{3} (1950) 1--6.
	
\bibitem{Bo_Lam_2001} T.Y. Lam, \textit{A First Course in Noncommutative Rings}, 2nd ed., GTM \textbf{131}, Springer-Verlag, New York, 2001.
	
\bibitem{Bo_Lennox_2004} J. C. Lennox and D. J. S. Robinson, \textit{The theory of infinite solvable groups}, Oxford, 2004.
	
\bibitem{Pa_Li_77} A. I. Lichtman, On subgroups of the multiplicative group of skew fields, \textit{Proc. Amer. Math. Soc.} \textbf{63}:1 (1977), 15--16.
	
\bibitem{Pa_Li_78} A. I. Lichtman, Free subgroups of normal subgroups of the multiplicative group of skew fields, \textit{Proc. Amer. Math. Soc.} \textbf{71}:2 (1978), 174--178.		
	
\bibitem{Bo_LyPa_2001} R. C. Lyndon and P. E. Schupp, \textit{Combinatorial group theory}, Springer, 2001.

\bibitem{Pa_mal'cev_1940} A. I. Mal'cev, On isomorphic matrix representations of infinite groups of matrices (Russian), \textit{Mat. Sb.} \textbf{8} (1940), 405–422, \textit{Amer. Math. Soc. Transl.} \textbf{45} (1965), 1–18
	
\bibitem{Pa_Milies_1981} C. P. Milies, Group rings whose torsion units form a subgroup II, \textit{Commun. Algebra,} \textbf{81}, 9 (1981), 699--712.
	
\bibitem{Pa_nbh_17} N. K. Ngoc, M. H. Bien,  B. X. Hai, Free subgroups in almost subnormal subgroups of general skew linear groups, \textit{Algebra i Analiz} \textbf{28}(5) (2016) 220-235, translation in \textit{St. Peters. Math. J.} \textbf{28}(5) (2017), 707--717.
	
\bibitem{Pa_Procesi_1966} C. Procesi, The Burnside problem, \textit{J. Algebra,} \textbf{4}, (1966), 421--425.

\bibitem{Bo_Ol'shanskii_1991} A. Yu Ol’shanskii, \textit{Geometry of Defining Relations in Groups}, Mathematics and its Applications, Vol. 70. Kluwer Academic Publishers, Dordrecht (1991).
	
\bibitem{Bo_SilberCoor_2010} T. C. Silberstein and M. Coornaert, \textit{Cellular Automata and Groups}, Springer-Verlag Berlin Heidelberg, 2010.
	
\bibitem{Bo_SilberBAdd_2021} T. C. Silberstein and M. D'Adderio, \textit{Topics in Groups and Geometry}, Springer Nature Switzerland AG, 2021.
	
\bibitem{Pa_Stuth_1964} C. J. Stuth, A generalization of the Cartan-Brauer-Hua Theorem, \textit{Proc. AMS}, \textbf{15}:2 (1964) 211--217. 
	
\bibitem{Pa_Schur_1911} I. Schur, Uber Gruppen periodischer Substitutionen, \textit{S. B. Preuss Akod. Wiss.} (1911), 619--627.

\bibitem{Pa_Scott_1957} W. R. Scott, On the multiplicative group of a division ring, \textit{Proc. AMS}, \textbf{8}, (1957), 303--305. 

\bibitem{Pa_ShWe_86} M. Shirvani and B. A. F. Wehrfritz, \textit{Skew Linear Groups}, Cambridge University Press, 1986.

\bibitem{Pa_Tits_1972} J. Tits, Free subgroups in linear groups, \textit{J. Algebra}, \textbf{20} (1972), 250--270.
	
\bibitem{Pa_Tokarenko_1968} A. Tokarenko, About linear groups over rings, \textit{Sib. Math. Zh,} \textbf{9}, (1968), 951--959.

\bibitem{Pa_Zalesskii_1965} A.E. Zalesskii, Solvable groups and crossed products, \textit{Mat. Sb.} (N.S.) \textbf{67} (109) (1) (1965) 154--160.
	
\bibitem{Zassenhaus_1938} H. Zassenhaus, Beweis eines satzes uber diskrete gruppen, \textit{Abh. Math. Sem. Univ. Hamburg.} \textbf{12} (1938), 289-312.

\bibitem{Zassenhaus_1969} H. Zassenhaus, On linear noetherian groups, \textit{J. Number Theory} \textbf{1} (1969), 70--89.
	
\end{thebibliography}
\end{document}